\documentclass{amsart}
\usepackage{amsmath,upref,amssymb,amsthm,amsxtra}
\usepackage{pdfsync}
\usepackage{mathtools}
\RequirePackage{calrsfs}
\usepackage[all]{xy}
\CompileMatrices
\usepackage{color}
\usepackage{multicol}
\usepackage{enumitem}
\usepackage{gensymb}

\usepackage{geometry}                % See geometry.pdf to learn the layout options. There are lots.
\usepackage{graphicx}

\usepackage{epstopdf}

\usepackage{MnSymbol}
\usepackage{tensor}

\RequirePackage{fix-cm} % arbitrary font scaling
\DeclareMathSizes{10}{10}{5}{4}

\theoremstyle{plain} % default
\newtheorem{theorem}[subsection]{Theorem}
\newtheorem{prop}[subsection]{Proposition}
\newtheorem{lemma}[subsection]{Lemma}
\newtheorem{corollary}[subsection]{Corollary}

\theoremstyle{remark}

\newtheorem*{rem*}{Remark}
\newtheorem*{rems*}{Remarks}

\newtheorem{exmp}[subsection]{Example}
\newtheorem*{note}{Note}
\theoremstyle{definition}
\newtheorem{definition}[subsection]{Definition}
\numberwithin{equation}{subsection}

\def\span{\mbox{Span}}
\newcommand{\ck}[1]{{#1}^{\bigvee}}
\newcommand{\pair}[1]{\langle #1\rangle}
\newcommand{\mpair}[1]{\pair{\,#1\,}}
\begin{document}

\title{On reflection representations of Coxeter groups\\over non-commutative rings}

\author {Annette Pilkington}
 \address{Department of Mathematics\\ University of Notre
Dame\\ Room 255 Hurley Building \\ Notre Dame, Indiana, 46556}
\email{Pilkington.4@nd.edu}


\begin{abstract} In this paper, we consider representations of Coxeter groups over a path algebra, R,  defined in Dyer \cite{Noncom}. We 
answer a question posed by Dyer about the multiplicative properties of R, showing that it is ``almost a domain". 
We also show that R cam be embedded in a matrix ring over a free product of extension fields of the rational numbers and rings of Laurent polynomials.
\end{abstract}

\maketitle 

\section*{Introduction}  Coxeter groups admit representations as certain discrete   reflection groups associated to root systems in real vector spaces (\cite{V}, \cite{Bour}, \cite{Hum}). In many situations, these have an associated analogue of  a Cartan matrix, describing the pairings between simple  roots in the  reflection representation, and simple coroots in the ``dual'' reflection representation. There are also reflection representations and root systems of corresponding Iwahori-Hecke algebras (\cite{CIK}).  Reflection  representations and root systems in modules over commutative rings have been considered in \cite{HeeRoot}, \cite{DyTh} and \cite{EmbI}-\cite{EmbII}. 

In \cite{Noncom}, similar notions of reflection representation  of Coxeter groups and Hecke algebras are considered over non-commutative rings. They are associated to a certain class of   matrices (non-commutative Cartan matrices, NCMs) with entries in non-commutative rings,   defined by identities related to Chebyshev polynomials  on the $2\times 2$-principal submatrices. We call these  ``lax'' reflection representations in this paper. They have the property that if $(W',S')$ is another Coxeter system and there is a group homomorphism $f\colon W\to W'$ which restricts to a bijection $S\to S'$, then lax reflection representations for $(W',S')$ pull back along $f$ to lax reflection representations of $(W,S)$.

An important role in the theory is played by a certain ring (the lax universal coefficient ring) with specified NCM, denoted in this paper as $\widetilde R=\widetilde R_{W,S}$, which gives (roughly speaking) an initial object in a   category of rings with a specified NCM for $(W,S)$. This ring, which is of considerable interest\footnote{It is stated without proof in \cite{Noncom}  that $\widetilde R$ can be identified as based ring with  the  two-sided ideal $J_{1}$ of Lusztig's equal parameter asymptotic Hecke  algebra $J$.}, may be defined as quotient of the path algebra of the Coxeter digraph by certain rank two relations. 
It contains a family $([e_{r}])_{r\in S}$ of non-zero orthogonal idempotents (the images in the quotient of length $0$ paths) 
such that  $\widetilde R=\bigoplus_{r,s\in S}[r]\widetilde R[s]$, and  is unital if and only if $S$ is finite.  The ring $\widetilde R$ is free as $\mathbb{Z}$-module and admits commuting left and right lax reflection $W$-actions, with the left (resp., right) action 
being by our conventions right (resp., left) $\widetilde R$-linear.   

This paper concerns a subclass of the 
lax reflection representations,  called here strict reflection 
representations, which is   defined by strict NCMs, satisfying stronger 
conditions than NCMs, and  for which the above pullback property no longer 
holds.  There is again a universal ring with strict NCM,  with  
analogous properties to those of $\widetilde R$ mentioned above. We denote this ring as $R=R_{W,S}$. 

We state a result which partly motivates the definition of $R$ and  clarifies the  relationship between lax and strict reflection representations, and  reflection representations   in real vector   spaces. For standard dual reflection representations of $(W,S)$ on real vector spaces $V$ and $\ck V$ with  bases of simple roots $\Pi$ and simple coroots $\ck \Pi$, there is an identification of  $\ck V\otimes_{\mathbb{R}}V$ with the ring $B$ of $S\times S$ real matrices with only finitely many non-zero entries (so that the matrix units are $\ck\alpha\otimes \beta$ for $\ck\alpha\in \ck \Pi$ and $\beta\in \Pi$) and a  ring homomorphism $R\to B$  which is equivariant for the left and right $W$-actions and  which induces a bi-equivariant ring surjection $R'=R\otimes_{\mathbb{Z}}\mathbb{R}\to B$ Here, we regard $W$ as acting on the left (resp., right) of $\ck V$ (resp., $V$).  This would be true with $R$ replaced by $\widetilde  R$, but that is of less interest since $R$ is a quotient (as ring with two-sided $W$-actions) of $\widetilde R$. For an irreducible finite Weyl group, $R'\to B$ is an isomorphism.  In this paper, we do not discus more delicate results of a similar nature to the last one for other   classes of Coxeter systems, or  extensions of the above facts to Hecke algebras.

 Although the class of strict reflection representations was not considered explicitly  in \cite{Noncom}, both  rings $R$  and  $\widetilde R$ are constructed there.  
   The  main result of this paper answers affirmatively  a basic question, raised in \cite[3.24]{Noncom},  of whether the strict universal coefficient ring $R$ is nearly a domain, in the sense that if $r,s,t\in S$ and $a\in [r]R[s]$ and $b\in [s]R[t]$, then $ab=0$ implies $a=0$ or $b=0$. The analogous statement   for $\widetilde R$ fails  whenever $\widetilde R\neq R$. One may hope that its validity for $R$ will make the $W$-actions on $R$  better behaved in some respects than those on $R$.
  We prove the main result by reducing to the case of finite $S$ and then  embedding $R$ as a subring of a ring  $\mathrm{Mat}_{S'\times S'}(A)$ of $S'\times S'$-matrices, where $S':=S\cup\{\bullet\}$ and  $A$ is a  $\mathbb{Q}$-algebra, depending on the Coxeter matrix,   which is a  free product (in the sense of Cohn \cite{Cohn1}), 
of certain real cyclotomic  number fields, Laurent polynomial rings and polynomial rings over $\mathbb{Q}$.

Since the techniques used come from diverse fields, some of which may be unfamiliar to the reader, we include basic definitions and cite basic results from 
 the various fields. In section \ref{intro}, we present  basic definitions and notation relating to rings and modules used throughout the paper. In section \ref{Graphs},
 we introduce Serre's notation  and terminology for graphs \cite{Serre} and in section \ref{palgs}, we give the basic definitions of path algebras and cite 
 a number of basic  results on Gr\"{o}bner bases of ideals in that context  from Green \cite{GreenA}. In section \ref{sec5}, we
 introduce  the path algebras  related to Coxeter systems and their quotients,  $\widetilde R_{W,S}$ and $R_{W,S}$,  appearing   in  Dyer \cite{Noncom}. In that section, we also  prove some basic results about these algebras, which allow us to restrict our attention to finite rank Coxeter groups and to 
 the ring $\hat{R}_{W,S} = R_{W,S}\otimes_\Bbb{Z}\Bbb{Q}$ when proving the main result. 
 In section \ref{secrealreps}, we provide for motivation some unpublished results of Dyer extending material from \cite{Noncom}
  on the lax and strict reflection representations associated to these rings and how they relate to reflection  representations on real vector spaces. We thank Matthew Dyer for his permission to include these results.  In section \ref{fpr}, we give the definition and results on free products of rings from 
 Cohn \cite{Cohn1}, \cite{Cohnfir}, and \cite{Cohn2} and we show that certain rings related to the edges of the underlying  path algebra are free products in the sense of Cohn. In section \ref{imbed}, we show how the ring $\hat{R}_{W,S}$ can be imbedded in a matrix rig over a free product of rings, We conclude with a 
 proof that the ring ${R}_{W,S}$ is nearly a domain,  as described above, in section \ref{conclude}.

\section{general Notation and Definitions}\label{intro}
\subsection{} The notation  $X = \{a, b, c \}$ will be used to denote that $X$ is a set with elements $a, b$ and $c$.
The set  $X\backslash Y$ will be the set of all elements in $X$ but not in $Y$ and the set $X \bigcupdot Z$ will denote the disjoint union of the sets $X$ and $Z$.

Unless otherwise stated, all rings considered have an identity (are unital).
If $R$ is a ring, we will denote the identity of $R$ by $1_R$ in  situations where the ring in question  may not be clear from the context. If $\phi :R_1 \to R_2$ is a homomorphism of rings and 
 $S \subset R_1$ is a subring of $R_1$, we let $\phi|_S$ denote the restriction of $\phi$ to $S$. If 
 $\psi :R_2 \to R_3$ is a homomorphism from $R_2$ to the ring $R_3$, we let $\psi \circ \phi$ denote the composition map.
We will use the notation $R_1 \cong R_2$ to indicate that the rings $R_1$ and $R_2$ are isomorphic.

If $X$ is a subset of a ring $R$, the two sided ideal of $R$ generated by the set $X$ will be denoted by $\langle X\rangle$.
If $I$ is an ideal of a ring $R$, we
write $r_1 \equiv r_2 \mod I$ when $r_1 - r_2 \in I$. 
We denote the ring of $n \times n$ matrices over a ring $R$ by $M_n(R)$. 
We will also use the notation $\langle x \in X | r \in R \rangle$ to denote a  group with generators in the set $X$ and relations in the set $R$.
In cases where the meaning of the notation is not clear from the context, we will add clarification.

The symbol $\mathbb{N}$ will be used to denote the 
non-negative integers,  $\{0, 1, 2, 3, \dots \}$, and the symbol $\mathbb{N}_{\geq k}$ will be used to denote the subset of $\mathbb{N}$ consisting of numbers greater than or equal to $k$. We let $[k]$ denote the set $\{1, \dots , k\}$.  The integers will be denoted by  $\mathbb{Z}$,  and  the rational numbers by $\mathbb{Q}$. 
Polynomial rings  $\mathbb{Q}[\{x_i\}_{i \in I}]$ will have non-commuting variables $\{x_i\}_{i \in I}$ unless stated otherwise.

\begin{definition} If $R$ is a ring and $a  \in R, a \not=0 $, we say that $a$ is a zero divisor in $R$ if there exists $b \in R, b \not= 0 $ such that $ab = 0$ or $ba = 0$. We say that $R$ is a domain if $R$ has no zero divisors. 
\end{definition}

Let  $A$ be a commutative unital ring. An (associative, possibly non-unital)  $A$-algebra is a (possibly non-unital, associative) ring  $B$ with an additional structure of unital $A$-module for which the multiplication map $B\times B\to B$ is $A$-bilinear (i.e. $(\alpha_1b_1 + \alpha_2b_2)(\alpha_3b_3 + \alpha_4b_4) = \displaystyle \sum_{i = 1}^2\sum_{j = 1}^2\alpha_i\alpha_jb_ib_j,$ for all $ \alpha_i \in A, b_j \in B$). Morphisms of $A$-algebras are $A$-linear ring homomorphisms. We call $A$ the coefficient ring of the $A$-algebra $B$. We regard  rings as $\mathbb{Z}$-algebras.

  A module $M$ for the $A$-algebra $B$ is a module for the ring $B$ together with a unital $A$-module structure on $M$ such that the multiplication map $B\times M\to M$ is $A$-bilinear.
   Morphisms of modules for the algebra $B$ are $B$-linear maps which are also $A$-linear.
    We write $\mathrm{End}_{B}(M)$ for the unital $A$-algebra of endomorphisms of $M$ and  $\mathrm{Aut}_{B}(M)$ for its unit group.  For a bimodule $M$ for a pair $(B,B')$ of $A$-algebras, it is required that the $A$-module structures from $B$ and $B'$ coincide.

\section{Graphs}\label{Graphs}
We use  the notation and terminology of Serre \cite{Serre} for graphs and trees.
\begin{definition}
A graph $\Gamma$ consists of a set of vertices $S = \mbox{Vert} \ \Gamma $, a set of edges  $Y = \mbox{Edge} \ \Gamma$ and two maps 
$$Y \to S \times S, \ \  y \to (o(y), t(y))$$
and 
$$Y \to Y, \ \ y \to \bar{y},$$
which satisfy the following condition: for each $y \in Y$, we have $\bar{\bar{y}} = y, \bar{y} \not= y$ and 
$o(y) = t(\bar{y})$. 
\end{definition}
Each $s \in S$ is called a vertex of $\Gamma$ and each $y \in Y$ is called an {\it edge} of $\Gamma$. 
If $y \in Y$, $o(y)$ is called the {\it origin} of $y$ and $t(y)$ is called the {\it terminus} of $y$, together $o(y)$ and $t(y)$ are called the {\it extremities }
of $y$. We say that two vertices are {\it adjacent}  if they are the extremities of some edge. 
\begin{definition} An orientation of a graph $\Gamma$ is a subset $Y_+$ of $Y = \mbox{Edge} \ \Gamma$  such that $Y$ is the disjoint union 
$Y = Y_+ \bigcupdot \overline{Y_+}$, where $\overline{Y_+} = \left\{\bar{y} | y \in Y_+ \right\}$.
\end{definition}

\subsection{Paths}
A {\it path of length $n \geq 1$} in a graph $\Gamma$ is a sequence of edges $p = (y_1, \dots , y_n)$ such that $t(y_i) = o(y_{i + 1})$, $1 \leq i \leq n - 1$.
We will  denote such a path by  $p = y_1 \dots y_n$. {\it Paths of length zero}  are just single vertices and will be denoted by $p = [s]$ for $s \in S$.
 The length of a path $p$ will be denoted by $\ell(p)$. 
 Let 
$\mathfrak{P}$ denote the set of all paths in $\Gamma$, 
we can extend the maps $o$ and $t$ to $\mathfrak {P}$ by letting $o( y_1 \dots  y_n ) = o(y_1)$ and $t( y_1 \dots  y_n ) = t(y_n)$ and $o([s]) = s = t([s])$. 
 If $p = y_1 \dots y_n$ with   $s_i = t(y_i)$ and $s_{i - 1} = o(y_i)$, we say that $p$ is a path from $s_0$ to $s_n$ and that  $o(y_1) = s_0$ and  $t(y_n) = s_n$ are the {\it extremities} of the path. 

A graph  is {\it connected} if any two vertices are the extremities of at least one path. If $p = y_1 \dots y_n$ and $y_{i + 1} = \bar{y}_i$
for some $1 \leq i \leq n - 1$, then the pair $(y_i, y_{i + 1})$ is called a {\it backtracking}.  A path $p$ is a {\it circuit} if it is a path 
 without backtracking such that $o(p) = t(p)$ and $\ell(p) \geq 1$. A circuit of length one  is called a loop. A graph $\Gamma$ is called {\it combinatorial} if it has no circuit of length less than or equal to  two.  If a graph $\Gamma$ is combinatorial, a path $p = y_1\dots y_n$ is determined by the extremities of its edges and can be characterized and denoted by its vertices as 
 $$p = y_1 \dots y_n = [s_0s_1 \dots s_n], \ \mbox{where} \ s_i = t(y_i), s_{i - 1} = o(y_i).$$
 Since our graphs of interest in this paper will be combinatorial, we will make frequent use of both notations for paths. In what follows, an edge may  be denoted by 
its label $y$ or it may be represented by its extremities as  $[o(y), t(y)]$.

 A {\it geometric edge} of a combinatorial graph, $\Gamma$, is a set $\{s_1, s_2\}$ of extremities of an edge in $\Gamma$. Each such geometric edge corresponds to a pair of edges $\{y, \bar{y}\}$ in the edge set of $\Gamma$.  A combinatorial graph is determined by 
 its vertices and geometric edges.

\begin{definition} let $\Gamma_1 = (S_1, Y_1)$ and $\Gamma_2 = (S_2, Y_2)$ be two graphs. A graph homomorphism  
 $\phi: \Gamma_1 \to \Gamma_2$ is a pair of maps $\phi_S: S_1 \to S_2$
 and $\phi_Y: Y_1 \to Y_2$ such that $o(\phi_Y(y)) = \phi_S(o(y))$ and 
 $t(\phi_Y(y)) = \phi_S(t(y))$ for every $y \in Y_1$. 
 \end{definition}
 
 \begin{exmp} Below we see a  representation of the combinatorial graph $\Gamma$, with  vertices, $S = \{r, s, t, u, v\}$,  and edges, $Y = Y_+ \bigcupdot \overline{Y_+}$,  where 
 $$Y_+ = \{y_\alpha = [rs], y_\epsilon = [st], y_\beta = [ut], y_\zeta = [tv], y_\delta = [uv], y_\gamma = [ru]\}$$
 We show only the directed edges in $Y_+$ on the graph.
 
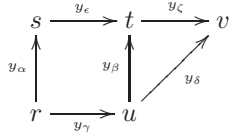
\begin{equation*}
\xymatrix{s\ar@{->}[r]^{y_{\epsilon}}&t\ar@{->}[r]^{y_\zeta}&{v}&&&\\
{r}\ar@{->}[u]^{y_\alpha}\ar@{->}[r]_{y_\gamma}&{u}\ar@{->}[u]^{y_\beta}\ar@{->}[ur]_{y_\delta}&&\\
}
\end{equation*}

 \end{exmp}

 The above example will serve as our running example throughout the paper.

\section{Path Algebras}\label{palgs}
In this section, we present the definition of a path algebra along with a summary of some related results on 
Gr\"{o}bner bases from Green \cite{GreenA}, which will prove useful in  later sections. 
 \subsection{} Let $\Gamma = \left(S, Y\right)$ be a combinatorial graph, with vertices $S = \{s_i \}_{i \in I}$ and edges $Y = \{y_j \}_{j \in J}$. 
Let  $\mathfrak {P}$ denote the set of  paths  in $\Gamma$ of finite length. 
Given a commutative ring $A$ with identity, the path algebra of $\Gamma$ over $A$, denoted $A\Gamma$, is an associative $A$-algebra, with an $A$-basis  of paths in $\mathfrak {P}$. Multiplication is given by concatenation;
$$[s_0s_1 \dots s_n][s'_0s'_1 \dots s'_m] = \begin{cases} 0 & \mbox{if} \ s_n \not= s'_0\\ 
[s_0s_1 \dots s_ns'_1 \dots s'_m] &  \mbox{if} \ s_n = s'_0 \end{cases}$$
Paths of length $0$ form a set of  orthogonal idempotents with this multiplication, and if $S$ is finite, $A\Gamma$ has an identity given by 
$$1 = \sum_{s \in S}[s].$$

\subsection{Admissible Orders on $\mathfrak {P}$ } For the remainder of  this section, we will restrict our discussion to the case where the base ring of our path algebra is a field, $K$,  and the case where the path algebra has finitely many vertices and edges, 
as in Green  \cite{GreenA} and Green et. al. in \cite{GHS}. We will combine results  from Green  \cite{GreenA} presented below with direct limits in subsequent  sections to ensure  that we can restrict our attention to the finite case when proving our main theorem. 

 A {\it well order} on the basis $\mathfrak {P}$ of  $K\Gamma$ is a total order on $\mathfrak {P}$, with the property that every nonempty subset of $\mathfrak {P}$ has a minimal element, or equivalently, for every descending chain of elements of $\mathfrak {P}$, $p_1 \geq p_2 \geq p_3 \geq \dots $, there exists some $N >0$ for which $p_N = p_{N + 1} = p_{N + 2} = \dots \ $.

Given a well order $>$ on $\mathfrak {P}$, and $w \not= 0  \in  K\Gamma$, with  $\displaystyle w = \sum_{i = 1}^na_ip_i, a_i \not= 0, a_i  \in K, p_i \in \mathfrak {P}$, we say that $p_i$ is  the {\it tip} of $w$, denoted $T(w)$, if $p_i \geq p_j$ for $1 \leq j \leq n$. If $X$ is a subset of $K\Gamma$, the tip of $X$ is 
$$T(X) = \{p \in \mathfrak {P} | p = T(x) \ \mbox{for some} \ x\not= 0, x \in X \}.$$
We denote  the set of {\em non tips} of $X$ by 
$$NT(X) = \mathfrak {P}\backslash T(X).$$
Note both T(X) and NT(X) depend on the order chosen. 

If   $W$ is a subspace of $K\Gamma$. By  Green \cite{GreenA}, Theorem 2.1, we have 
\begin{equation}\label{nform}
K\Gamma = W \oplus \mbox{Span}(NT(W)).
\end{equation}

We say that an order on $\mathfrak {P}$ is {\em admissible} if it is a well order and satisfies the following conditions for $r, s, t, u \in \mathfrak {P}$:
\begin{enumerate}
\item if $r < s$ then $rt < st$ if both $rt \not= 0$ and $st \not= 0$.
\item if $r < s$ then $ur < us$ if both $ur \not= 0$ and $us \not= 0$.
\item if $r = st \not= 0$ then $r \geq s$ and $r \geq t$.
\end{enumerate}

A number of admissible well orders can be imposed on the paths in a  path algebra, see Green \cite{GreenA}. We give  one example of such an admissible ordering  below.

\begin{exmp} The (left) length lexicographic order on $\mathfrak {P}$:\\
 Let $\Gamma = \left(S, Y\right)$ be a  graph, with vertices $S = \{s_i \}_{i \in [n]}$ and edges $Y = \{y_j \}_{j \in [m]}$. 
Let  $\mathfrak {P}$ denote the set of  paths  in $\Gamma$ of finite length. 
We choose an ordering on paths of length $\leq 1$, in such a way that each element of $S$ is less that each element of $Y$. 
$$[s_1] < [s_2] < \dots < [s_n] < y_1 < y_2 < \dots < y_m.$$
If $p$ and $q$ are paths of length at least one, $p < q$ if $l(p) < l(q)$, or if $l(p) = l(q)$ and $p = y_1y_2 \dots y_iy_{i+1} \dots y_n$, $q = y_1y_2 \dots y_iy'_{i+1} \dots y'_n$ with $y_{i + 1} < y'_{i + 1}$.
It is not difficult to check that this is in in fact a well order and conditions (1), (2) and (3) in the definition of an admissible well order hold. 
\end{exmp}

\subsection{Gr\"{o}bner Basis of an Ideal}
If $>$ is an admissible order on $\mathfrak {P}$,  and $I$ an ideal in  $K\Gamma$ ($K$ a field), we say that a subset $\mathfrak {G} \subset I$  is {\em a Gr\"{o}bner basis for the ideal $I$, with respect to $>$}, if the two-sided ideal generated by $T(\mathfrak {G})$ is equal to the two-sided ideal generated by $T(I)$. Equivalently, $\mathfrak {G} \subset I$ is a Gr\"{o}bner basis for $I$ if for every
$b \in T(I)$ there is some $g \in \mathfrak {G}$ such that  $b = pT(g)q$ for some $p, q \in \mathfrak {P}$.

Given a Gr\"{o}bner basis, $\mathfrak {G}$,  for an ideal $I$ of  $K\Gamma$ with respect to an admissible order $>$, we have, by Equation \ref{nform}, that as a vector space over $K$,

\begin{equation} \label{gbasis}
K\Gamma = I \oplus \span(NT(I)) = \langle \mathfrak {G} \rangle \oplus \span(NT(I)).
\end{equation}

  Equation \ref{gbasis} allows us to use a   Gr\"{o}bner basis, $\mathfrak {G}$,  to determine ideal membership for $I$.
  Given a Gr\"{o}bner basis, $\mathfrak {G}$,  for an ideal $I$ of  $K\Gamma$, each $y \in K\Gamma$ has a unique {\em normal form} $N(y) \in \span(NT(I))$ with respect to the order $>$, where $y = y_1 + N(y)$ and $y_1 \in I$.  
Since $\left\langle T(\mathfrak {G})\right\rangle = \left\langle T(I)\right\rangle$, we have that $\span(NT(I))$ has a basis consisting of those paths in $\mathfrak {P}$ which are not divisible by 
$T(g)$ for some $g \in \mathfrak {G}$ (i.e. the paths in $K\Gamma$ which are not of the form $pT(g)q$ for some  $g \in \mathfrak {G}$, $p, q \in \mathfrak {P}$). 
The following is  special case of  Proposition 2.9, Green \cite{GreenA}, which follows easily from Equation \ref{gbasis} and makes explicit the relationship between $K\Gamma/I$ and $\span(NT(I))$:
\begin{prop}\label{prop}(See  Proposition 2.9, Green \cite{GreenA}) Let $\Gamma$ be a graph with finitely many vertices and edges. Let $K\Gamma$ be the corresponding  path algebra with basis of paths, $\mathfrak {P}$, admitting  an admissible order $>$. Let $I$ be an ideal of $K\Gamma$ with Gr\"{o}bner basis $\mathfrak {G}$ and let $x, y \in K\Gamma$ with 
equivalence classes  $x + I$ and  $y + I$  respectively in the quotient algebra $K\Gamma/I$. Then
\begin{enumerate} 
\item $x + I = y + I$  if and only if $N(x) = N(y)$.
\item $x + I = N(x) + I$. 
\item The map $\sigma: K\Gamma/I \to K\Gamma$ with $\sigma(x + I) = N(x)$ is a vector space splitting to the canonical surjection 
$\pi: K\Gamma \to K\Gamma/I$. 
\item $\sigma$ induces  a $K$-linear isomorphism between $K\Gamma/I$ and $\span(NT(I))$.
\item  Identifying $K\Gamma/I$ with $\span(NT(I))$, then $NT(I)$ is a K-basis of $K\Gamma/I$ contained in $\mathfrak {P}$. 
\end{enumerate}
\end{prop}

\subsection{Overlap Relations }
Let $>$ be an admissible order on $\mathfrak {P}$. If $f, g \in K\Gamma$, $f, g \not= 0$,  and $p, q \in \mathfrak {P}$ such that 
\begin{enumerate}
\item $T(f)q = pT(g)$.
\item $T(f)$ does not divide $p$ and $T(g)$ does not divide $q$.
\end{enumerate}
The {\em overlap relation} of $f$ and $g$ by $p, q$ is 
$$o(f, g, p, q) = (1/C(T(f)))fq - (1/C(T(g)))pg,$$
where $C(T(h))$ denotes  the coefficient of $T(h)$ in $h   \in K\Gamma.$ 
Note that $T(o(f, g, p, q)) < T(f)q = pT(g)$.

We say the elements of a set $\mathfrak {G}$ is {\em tip reduced} if for distinct elements $g_1, g_2 \in \mathfrak {G}$, $T(g_1)$ does not divide $T(g_2)$. An element of the  path algebra $K\Gamma$, $\sum_{i = 1}^{n}a_ip_i, a_i \in A, a_i \not= 0, p_i \in \mathfrak {P}$, is said to be {\em left uniform} if for each $p \in \mathfrak {P}$,  either $pp_i = 0$ for all $i$ or $pp_i \not= 0$ for all $i$, $1 \leq i \leq n$. Right uniform elements are defined similarly and  an element is {\em uniform} if it is both left and right uniform.
An element of the  path algebra $K\Gamma$,   $\sum_{i = 1}^{n}a_ip_i, a_i \in K, a_i \not= 0,  p_i \in \mathfrak {P}$ is uniform if and only if there are vertices 
$v, w$ such that for $1 \leq i \leq n$, $o(p_i) = v$ and $t(p_i) = w$. The following result,  limited for simplicity  to  a special case of   Theorem 2.3,  Green 
 \cite{GreenA} which is sufficient for our needs, 
can be used to identify when a set of polynomials form a  Gr\"{o}bner basis for an ideal in a path algebra:

\begin{theorem}(See Green \cite{GreenA}, Theorem 2.3) \label{overlap} Let $\Gamma$ be a graph with finitely many vertices and edges. 
 Suppose that $\mathfrak {G}$ is a set of uniform, tip reduced elements of $K\Gamma$. If for every overlap relation,
 $o(g_1, g_2, p, q)$, with $g_1, g_2 \in \mathfrak {G}$, $p, q \in \mathfrak {P}$, we have;
$$o(g_1, g_2, p, q) = 0,$$
then $\mathfrak {G}$ is a Gr\"{o}bner basis for $\langle\mathfrak {G}\rangle$.
\end{theorem}

\section{Coxeter Systems and Associated Path Algebras}\label{sec5}
In this section, we introduce   the Path Algebras related to Coxeter systems and the rings $R= R_{W, S}$ and $\widetilde{R} = \widetilde{R}_{W, S}$ from  Dyer \cite{Noncom}, discussed in the introduction. We also prove a number of lemmas which will be used in later sections to simplify the proof of our main result. 
\subsection{Path Algebras associated to Coxeter Systems}
 \label{cox}A {\em Coxeter matrix} is a matrix $(m_{ij})_{i, j \in I}$, where $I$ is a set of indices, with 
$m_{ii} = 1$ for all $i \in I$ and $m_{ij} = m_{ji} \in \mathbb{N}_{\geq 2}\cup \{\infty\}$ for all $i, j \in I, i \not= j$. 
If $S = \{s_i\}_{i \in I}$ is a set indexed by $I$, the corresponding {\em Coxeter group} is the group with 
presentation
$$W = \langle s_i (i \in I) | (s_ks_l)^{m_{kl}} = 1 \ \mbox{for} \ s_k, s_l \in S \ \mbox{with} \ m_{kl} \not= \infty \rangle.$$
The pair $(W, S)$ is called a {\em Coxeter system} with Coxeter matrix $(m_{ij})_{i, j \in I}$. We say that $(W, S)$ has {\em finite rank} if $S$ is finite.

We associate to a  Coxeter system $(W, S)$ a graph $\Gamma_{W, S}$ with  vertices $s_i \in S$  and 
 edges  $Y_{W, S} = \left\{y_{ij} = [s_is_j]  |\  i \not= j \ \mbox{and} \  m_{ij} \geq 3\right\}$. 
 Let $\mathfrak {P}_{W, S}$ denote the set of paths in $\Gamma_{W, S}$. We let  $P_{W, S}$ denote the path algebra $\mathbb{Z}\Gamma_{W, S}$ and let $\hat{P}_{W, S}$ denote the path algebra $\mathbb{Q}\Gamma_{W, S}$. 

For $m \in \mathbb{Z}, m \geq 3$, let $C_m(t)$ denote the  minimum polynomial of $\displaystyle 4\cos^2\frac{\pi}{m}$ in $\mathbb{Z}[t]$. We have 
$$C_{3}(t) = t - 1, \ 
C_4(t) = t - 2,  \ 
C_5(t) = t^2 -3t + 1, \ 
C_6(t) = t - 3, \dots . $$
To each edge $y_{ij} = [s_is_j]$ of $\Gamma_{W, S}$ where $3 \leq m_{ij} < \infty$, we associate an element  $C_{y_{ij}}$ of the 
subring $[s_i]P_{W, S}[s_i]\subset  {P}_{W, S} $ (resp. $[s_i] \hat{P}_{W, S}[s_i] \subset \hat{P}_{W, S}$), where the imbedding is non-unital. 
Notice that the ring  $[s_i]P_{W, S}[s_i]$ (resp. $[s_i] \hat{P}_{W, S}[s_i]$) has identity $[s_i]$, making  evaluation of the above polynomials  possible over the ring. We let 
$$C_{y_{ij}} = C_{m_{ij}}([s_is_js_i]) = C_{m_{ij}}(y_{ij}\bar{y}_{ij}),$$ 
with $C_{m_{ij}}$ viewed as a polynomial over the ring $[s_i]P_{W, S}[s_i]$ (resp. $[s_i] \hat{P}_{W, S}[s_i]$). 
 Let $I_{W, S}$, (respectively $\hat{I}_{W, S}$),  be the ideal 
  of $P_{W, S}$ (resp. $\hat{P}_{W, S}$)
 generated by the set $\mathfrak {C}_{W, S} = \left\{C_y |  y \in Y_{W, S}\right\}$, and  let 
  $R_{W, S} = P_{W, S}/I_{W,S}$ (resp.  $\hat{R}_{W, S} = \hat{P}_{W, S}/\hat{I}_{W, S}$) denote the  quotient algebras.

 \begin{lemma} \label{grob} Let $\Gamma_{W, S}$ be a  graph associated to a finite rank Coxeter System $(W, S)$ with Coxeter matrix $(m_{ij})_{\{i, j \in I\}}$.   Then  $\mathfrak {C}_{W, S}$  is a Gr\"{o}bner basis for the ideal  $\hat{I}_{W, S}$ in the path algebra $\hat{P}_{W,S}$.
 \end{lemma}
 \begin{proof} Regardless of the admissible ordering used, $T(C_{y_{ij}}) = (y_{ij}\bar{y}_{ij})^{d_{ij}}$, where degree $C_{m_{ij}} = d_{ij}$, for 
 all $y_{ij} \in Y_{W, S}$ with $m_{ij} < \infty$. 
Clearly  the set $\mathfrak {C}_{W, S}$ is tip reduced and uniform, therefore, by Theorem \ref{overlap}, we need only verify that 
 all overlap relations $o(C_{y_{ij}}, C_{y_{kl}}, p, q)$, $y_{ij}, y_{kl}  \in Y, m_{ij}, m_{kl} < \infty, p, q \in \mathfrak {P}_{W, S}$ are zero. Consider such an overlap relation, $T(C_{y_{ij}})q = pT(C_{y_{kl}})$, for $p, q \in \mathfrak {P}$, where $T(C_{y_{kl}})$ does not divide $q$ and $T(C_{y_{ij}})$ does not divide $p$. Since $T(C_{y_{ij}}) = (y_{ij}\bar{y}_{ij})^{d_{ij}}$  and 
 $T(C_{y_{kl}}) = (y_{kl}\bar{y}_{kl})^{d_{kl}}$, and $T(C_{y_{kl}})$ does not divide $q$, we must have 
 $T(C_{y_{kl}}) = aq$ for some $a \in \mathfrak {P}_{W, S}$ where $\ell(a) \geq 1$. Hence 
 $$T(C_{y_{ij}})q = (y_{ij}\bar{y}_{ij})^{d_{ij}}q = pT(C_{y_{kl}}) = paq = p(y_{kl}\bar{y}_{kl})^{d_{kl}},$$
  and since $a$ divides both $(y_{ij}\bar{y}_{ij})^{d_{ij}}$ and $(y_{kl}\bar{y}_{kl})^{d_{kl}}$, we must have that either  $y_{kl} = y_{ij}$ or $y_{kl} = \bar{y}_{ij}$.  If $y_{kl} = y_{ij}$, then
 $(y_{ij}\bar{y}_{ij})^{d_{ij}}q =  p(y_{ij}\bar{y}_{ij})^{d_{ij}}$ and by comparing lengths and edges, we see 
 that $p = q = (y_{ij}\bar{y}_{ij})^{k}$ for some $k$.  Thus  $C_{y_{ij}}p - qC_{y_{ij}} = 0$ in this case.
  If $y_{kl} = \bar{y}_{ij}$, then $(y_{ij}\bar{y}_{ij})^{d_{ij}}q =  p(\bar{y}_{ij}{y}_{ij})^{d_{ij}}$. In this case,
  a comparison of lengths and edges shows that $p = q = (y_{ij}\bar{y}_{ij})^{k}y_{ij}$ for some $k$ and 
  it is clear that  $C_{y_{ij}}p - qC_{\bar{y}_{ij}} = 0$. This completes our proof.
  \end{proof}

\subsection{}Let $J \subset I$ and $S_J = \{s_j | j \in J\}$. Let $(W_J, S_J)$ denote the Coxeter system with Coxeter matrix $(m_{ij})_{i, j \in J}$. The group 
$W_J$ has generators and relations
$$W = \langle s_j (j \in J) | (s_ks_l)^{m_{kl}} = 1 \ \mbox{for} \ s_k, s_l \in S_J \ \mbox{with} \ m_{kl} \not= \infty \rangle.$$
and is called 
a {\em standard parabolic subgroup} of $W$. 
If $J_1 \subseteq J_2 \subseteq I$ we have an inclusion map of graphs 
from  the vertices $S_{J_1}$ and edges $Y_{W_{J_1}, S_{J_1}}$ in 
 $\Gamma_{W_{J_1}, S_{J_1}}$ to  the corresponding vertices and edges in  $\Gamma_{W_{J_2}, S_{J_2}}$. 
  This  extends to an inclusion map 
 $i_{J_1, J_2}: \mathfrak {P}_{W_{J_1}, S_{J_1}} \to  \mathfrak {P}_{W_{J_2}, S_{J_2}}$, which in turn gives us a non-unital  monomorphism of $\mathbb{Q}$ algebras; 
 $\hat{i}_{J_1, J_2}: \hat{P}_{W_{J_1}, S_{J_1}} \to  \hat{P}_{W_{J_2}, S_{J_2}}$. Since $\hat{i}_{J_1, J_2}(\mathfrak {C}_{W_{J_1}, S_{J_1}})
 \subseteq \mathfrak {C}_{W_{J_2}, S_{J_2}}$, we get a well defined induced non-unital  homomorphism of quotient algebras:
 $$\tilde{i}_{J_1, J_2}: \hat{R}_{W_{J_1}, S_{J_1}} \to  \hat{R}_{W_{J_2}, S_{J_2}}.$$
 
 \begin{lemma} \label{impar}Let $(W, S = \{s_i\}_{i \in I})$ be a Coxeter system.  Let $J_1, J_2$ be finite  subsets of $I$ with $J_1 \subseteq J_2$, then the map 
   $\tilde{i}_{J_1, J_2}: \hat{R}_{W_{J_1}, S_{J_1}} \to  \hat{R}_{W_{J_2}, S_{J_2}}$ 
   is a monomorphism of non-unital $\mathbb{Q}$-algebras. 
   \end{lemma}
   \begin{proof} Let $>$ be an admissible ordering on $\mathfrak {P}_{W_{J_2}, S_{J_2}}$. The restriction of $>$ to $i_{J_1, J_2}(\mathfrak {P}_{W_{J_1}, S_{J_1}})$,
   gives us an admissible ordering on $\mathfrak {P}_{W_{J_1}, S_{J_1}}$ which, for the sake of simplicity,  we will also denote by $>$.   
   By Proposition \ref{prop}, $\hat{R}_{W_{J_i}, S_{J_i}}$ can be identified as a vector space over $\mathbb{Q}$ with $\span(NT(\hat{I}_{W_{J_i}, S_{J_i}})), i \in \{1, 2\}$ (with respect to the ordering $>$).   If $p \in NT(\hat{I}_{W_{J_1}, S_{J_1}})$, then $p$ is not divisible by $T(C_y)$ for any $C_y \in \mathfrak {C}_{W_{J_1}, S_{J_1}}$. Since $p \in \mathfrak {P}_{W_{J_1}, S_{J_1}}$, $p$ cannot be divisible by $T(C_y)$ for any $C_y \in \mathfrak {C}_{W_{J_2}, S_{J_2}}\backslash\mathfrak {C}_{W_{J_1}, S_{J_1}}$ and thus $i_{J_1, J_2}(p) \in NT(\hat{I}_{W_{J_2}, S_{J_2}})$. Thus, the map 
    $i_{J_1, J_2}: \mathfrak {P}_{W_{J_1}, S_{J_1}} \to  \mathfrak {P}_{W_{J_2}, S_{J_2}}$ restricts to a one-to-one map 
     $i'_{J_1, J_2}: NT(\hat{I}_{W_{J_1}, S_{J_1}}) \to  NT(\hat{I}_{W_{J_2}, S_{J_2}})$. Using the vector space identification 
     above, we see that  $\tilde{i}_{J_1, J_2}: \hat{R}_{W_{J_1}, S_{J_1}} \to  \hat{R}_{W_{J_2}, S_{J_2}}$ 
   is a monomorphism as desired.
   \end{proof}
   \subsection{Direct Systems} We shall prove some of our  results involving rings attached to  infinite rank Coxeter groups by regarding them  as direct limits of rings attached to the finite rank standard parabolic subgroups.
   \begin{definition} A {\em directed poset} is a poset $(P, \leq)$, $P \not= \emptyset$,  such that  for any $p_1, p_2 \in P$, there is an element $p_3 \in P$ such that 
$p_1 \leq p_3$ and $p_2 \leq p_3$.  
\end{definition}

\begin{definition}\label{dsyst} A {\em direct system} $\left\{P, \{R_p\}, \{i_{p_1, p_2}\}\right\}$ in the category of non-unital $\mathbb{Q}$-algebras consists of a directed poset,  $P$, a collection of non-unital 
$\mathbb{Q}$-algebras, $\{R_p\}_{p \in P}$, and a set of morphisms of $\mathbb{Q}$ algebras, $\{i_{p_1, p_2}: R_{p_1} \to R_{p_2}\}_{p_1, p_2 \in P, p_1 \leq p_2}$ such that 
\begin{enumerate}
\item $i_{p, p }$ acts as the identity on $R_p$.
\item $i_{p_2, p_3} \circ i_{p_1, p_2} = i_{p_1, p_3}$ for $p_1, p_2, p_3  \in P,  \ p_1 \leq p_2 \leq p_3$.
\end{enumerate}
\end{definition}

\begin{note} When considering direct systems in the category of non-unital $\mathbb{Q}$-algebras throughout the remainder of this section, all associated morphisms considered will 
be morphisms in that category. 
\end{note}

\begin{definition}\label{umap} Let $\left\{P, \{R_p\}, \{i_{p_1, p_2}\}\right\}$ be a direct system of non-unital $\mathbb{Q}$ algebras. A 
$\mathbb{Q}$ algebra $R$ is called the direct limit of the direct system, denoted $ \varinjlim{R_p}$, 
if there exist  homomorphisms, $i_p: R_p \to R,\  p \in P$ such that 
\begin{enumerate}
\item For any  $p_1, p_2 \in P$, $p_1 \leq p_2$, the following diagram commutes:
 \begin{equation*}
\xymatrix{R_{p_1}\ar@{->}[dr]^{i_{p_1}}\ar@{->}[d]_{i_{p_1, p_2}}\\
{R_{p_2}}\ar@{->}[r]_{i_{p_2}}&{R}\\
}
\end{equation*}
\item Given a non-unital $\mathbb{Q}$ algebra, $Q$, and a set of  homomorphisms
$\phi_p : R_p \to Q$, $p \in P$, such that  the diagram 
 \begin{equation*}
\xymatrix{R_{p_1}\ar@{->}[dr]^{\phi_{p_1}}\ar@{->}[d]_{i_{p_1, p_2}}\\
{R_{p_2}}\ar@{->}[r]_{\phi_{p_2}}&{Q}\\
}
\end{equation*}
commutes for all $p_1, p_2 \in P$, $p_1 \leq p_2$, then there exists a unique  homomorphism $\phi: R \to Q$ such that the following diagram commutes for all $p \in P$:
 \begin{equation*}
\xymatrix{R_{p}\ar@{->}[dr]_{\phi_{p}}\ar@{->}[r]^{i_{p}}&R\ar@{->}[d]^{\phi}\\
&{Q}\\
}
\end{equation*}
%HERE
\end{enumerate}
\end{definition}

\begin{lemma} \label{const} Let $\left\{P, \{R_p\}, \{i_{p_1, p_2}\}\right\}$ be a direct system of non-unital $\mathbb{Q}$ algebras. If $i_{p_1, p_2}$
is a monomorphism for all $p_1, p_2 \in P$, then the homomorphisms $i_p : R_p \to  \varinjlim{R_p}$ making the diagram below commute, are also monomorphisms; 
 \begin{equation}\label{morph}
\xymatrix{R_{p_1}\ar@{->}[dr]^{i_{p_1}}\ar@{->}[d]_{i_{p_1, p_2}}\\
{R_{p_2}}\ar@{->}[r]_{i_{p_2}}&{ \varinjlim{R_p}}\\
}
\end{equation}
\end{lemma}

\begin{proof}One can construct the direct limit of non-unital $\mathbb{Q}$ algebras as follows. 
Let 
$$R =  \varinjlim{R_p} = \bigcupdot_{P \in P} R_p/\sim$$
where $\sim$ is the equivalence relation on $\bigcupdot_{P \in P} R_p$ defined by 
$$r_{1} \sim r_{2}, r_1 \in R_{p_1}, r_2 \in R_{p_2} \ \mbox{iff.} \ i_{p_1, p_3}(r_1) =  i_{p_2, p_3}(r_2) \ \mbox{for some} \ p_3 > p_1, p_2.$$
The ring and algebra operations on the equivalence classes of  $R$, denoted by square brackets $[ ]$,  are defined as follows.
For each $\alpha \in \mathbb{Q}$ and each $[r] \in R$, let $\alpha[r] = [\alpha r]$. 
 Given $[r_1], [r_2] \in R$ such that $r_1 \in R_{p_1}$
and $r_2 \in R_{P_2}$, then for $p_3 > p_1, p_2$, we define $[r_1] + [r_2]  = [i_{p_1, p_3}(r_1) + i_{p_2, p_3}(r_2)]$ and 
$[r_1][r_2]  = [i_{p_1, p_3}(r_1) i_{p_2, p_3}(r_2)]$. The equivalence class $[0]$ gives a zero element for the ring $R$. 
It can be checked easily that R along with the homomorphisms $i_p:R_p \to R, p \in P$ given by $i_p(r) = [r]$ for $r \in R_p$
form a non-unital $\mathbb{Q}$ algebra making the diagram \ref{morph} above commute.
Given a non-unital $\mathbb{Q}$ algebra $Q$ and a set of morphisms $\phi_p: R_p \to Q$, $p \in P$, such that $\phi_{p_2}\circ i_{p_1, p_2} = 
\phi_{p_1}$, we define $\phi: R \to Q$ as $\phi([r_p]) = \phi_p(r_p)$ for $r_p \in R_p$. One can easily check that $\phi$ is a well defined homomorphism with the property that $\phi \circ i_p = \phi_p$ for all $p \in P$. It is unique since any homomorphism $\psi : R \to Q$ with 
$\psi \circ i_p = \phi_p$ must agree with $\phi$.

 By the universal mapping property, the direct limit is unique up to isomorphism, hence it is enough to show that the homomorphisms  $i_p: R_p \to R$ are monomorphism for all $p  \in P$  for the non-unital $\mathbb{Q}$ algebra $R$  constructed above.
Now it is clear that if $i_p(r) = [0]$ for some $r \in R_p, p \in P$, then we must have $i_{p, p_2}(r)  = 0 \in R_{p_2}$ for some 
$p_2 \in P$. Thus if all of the maps  $i_{p_1, p_2}, p_1, p_2 \in P$ are monomorphisms, we must have that each $i_p, p \in P$ is also a monomorphism.
\end{proof}

\subsection{} For the  Coxeter system $(W, S = \{s_i\}_{i \in I})$, let  $\mathcal{P}(I)$ denote the set of finite subsets of $I$. Then $\mathcal{P}(I)$  is a directed poset with inclusion as a partial ordering. It is not hard to verify that 
the quotient algebras associated to the  family of parabolic subgroups,  $\{\hat{R}_{W_J, S_J}\}_{J \in \mathcal{P}(I)}$, and the  homomorphisms 
$\tilde{i}_{J_1, J_2}:\hat{R}_{W_{J_1}, S_{J_1}} \to  \hat{R}_{W_{J_2}, S_{J_2}},$ $J_1 \subseteq J_2$,  form a direct system, $\left\{\mathcal{P}(I), \{\hat{R}_{W_J, S_J}\}, \{\tilde{i}_{J_1, J_2}\}\right\}$
 in the category of non-unital algebras over $\mathbb{Q}$.

\begin{lemma}\label{dirlim} Let $(W, S = \{s_i\}_{i \in I})$, be a Coxeter system and $\left\{\mathcal{P}(I), \{\hat{R}_{W_J, S_J}\}, \{\tilde{i}_{J_1, J_2}\}\right\}$ the  direct system  of quotient algebras associated to  the  parabolic subgroups of $W$ of finite rank. Then 
\begin{enumerate}
\item Given $x \in \hat{R}_{W, S}$, there exists $J \in \mathcal{P}(I)$ such that $x = \tilde{i}_{J, I}(x')$ for some $x' \in \hat{R}_{W_J, S_J}$.
\item The homomorphisms  $\tilde{i}_{J, I}: \hat{R}_{W_J, S_J} \to \hat{R}_{W, S}$ have the property that  $\tilde{i}_{J_2, I} \circ \tilde{i}_{J_1, J_2} = \tilde{i}_{J_1, I}$ for all $J_1, J_2 \in \mathcal{P}(I)$, 
$J_1 \subseteq  J_2$. 
\item Given a non-unital $\mathbb{Q}$ algebra $Q$ and a family of  homomorphisms $\{\phi_J : \hat{R}_{W_{J}, S_{J}} \to Q | J \in \mathcal{P}(I)\}$ such that the following diagram commutes for each $J_1 \subseteq  J_2 \subset I$, $J_1$ and $J_2$ finite;
 \begin{equation} \label{Qmaps}
\xymatrix{ \hat{R}_{W_{J_1}, S_{J_1}}\ar@{->}[dr]^{\phi_{J_1}}\ar@{->}[d]_{\tilde{i}_{J_1, J_2} }\\
{ \hat{R}_{W_{J_2}, S_{J_2}}}\ar@{->}[r]_{\phi_{J_2}}&{Q}\\
}
\end{equation}
then there exists a  homomorphism $\phi$ such that  the following diagram commutes for all finite subsets $J$ of $I$:
 \begin{equation} \label{Q}
\xymatrix{\hat{R}_{W_{J}, S_{J}}\ar@{->}[dr]_{\phi_{J}}\ar@{->}[r]^{ \tilde{i}_{J, I} }&\hat{R}_{W, S}\ar@{->}[d]^{\phi}\\
&{Q}\\
}
\end{equation}
\item $\hat{R}_{W, S} = \varinjlim{\hat{R}_{W_J, S_J}}$
in the category of non-unital $\mathbb{Q}$ algebras. 
\item  The maps $\tilde{i}_{J, I}: \hat{R}_{W_J, S_J} \to \hat{R}_{W, S}$ are  injective.
\end{enumerate}
\end{lemma}

\begin{proof}
Given an $x = p + \hat{I}_{W, S} \in \hat{R}_{W, S}$, where $p \in P_{W, S}$,  then $p = \sum_1^n\alpha_ip_i$ where $p_i \in \mathfrak {P}_{W, S}$. 
Since each path $p_i$ is a product of only a finite number of edges, we have $x = \tilde{i}_{J, I}(x')$ for 
some $J \in \mathcal{P}(I)$ and $x' \in \hat{R}_{W_{J}, S_{J}}$. This proves (1).

Since, $(\tilde{i}_{J_2, I} \circ \tilde{i}_{J_1, J_2})(p + \hat{I}_{W_{J_1}, S_{J_1}}) = \tilde{i}_{J_1, I}(p)$ for all $J_1, J_2 \in \mathcal{P}(I)$, 
$J_1 \subseteq  J_2$, and all $p \in \mathfrak {P}_{W_{J_1}, S_{J_1}}$, (2) follows.

Assume now that we are given a non-unital $\mathbb{Q}$ algebra $Q$ and a family of  homomorphisms $\{\phi_J : \hat{R}_{W_{J}, S_{J}} \to Q | J \in \mathcal{P}(I)\}$ making  Diagram  \ref{Qmaps} commute. Let $x = p + \hat{I}_{W, S} \in \hat{R}_{W, S}$. By (1),  $x = \tilde{i}_{J, I}(x')$ for 
some $J \in \mathcal{P}(I)$ and $x' \in \hat{R}_{W_{J}, S_{J}}$. We let $\phi(x) = \phi_J(x')$. 
By (2) and the commutativity of Diagram \ref{Qmaps}, $\phi(x)$ is independent of the choice of $J$. For  each $C_y \in \mathfrak {C}_{W, S}$, $\phi(C_y + \hat{I}_{W, S}) = \tilde{i}_{J, I}(C_y + \hat{I}_{W_J, S_J}) =  \tilde{i}_{J, I}(0) = 0$, where $J = \{i, j | o(y) = s_i, t(y) = s_j\}$. Thus $\phi$ is well defined. That the diagram \ref{Q} commutes follows from the definition of $\phi$, and since any  homomorphism  making \ref{Q} commute must agree with $\phi$, we see it is the unique homomorphism with this property.  This establishes (3). 

From Definition \ref{umap}, we see that (4) follows from (2) and (3). 
By Lemma \ref{impar}, the maps $\tilde{i}_{J_1, J_2}$ in our direct system are monomorphisms. Therefore  (5) follows from   (2), (3), (4) and Lemma \ref{const}. 
This completes the proof.
\end{proof}

\begin{lemma}\label{baaisN} Let $(W, S)$ be a Coxeter system and $\Gamma_{W, S}$, the associated graph.  Let   $\hat{\pi}$ denote the quotient 
map $\hat{P}_{W,S}$ to $\hat{R}_{W,S}$. Let $\mathfrak {N} = \{x \in \mathfrak {P}_{W, S} | x \not= ptq, t \in T(\mathfrak {C}_{W, S}), p, q \in \mathfrak {P}_{W, S}\}$. 
The 
set $\hat{\pi}(\mathfrak {N})$ forms a basis for $\hat{R}_{W, S}$ as a vector space over $\mathbb{Q}$.
\end{lemma}

\begin{proof}
Let  $x \in \hat{R}_{W, S}$, by Lemma \ref{dirlim} (1), we have $x = \tilde{i}_{J, I}(r)$ for some finite $J \subseteq I$ and   $r \in \hat{R}_{W_{J}, S_{J}}$. Hence $x \in \span(\tilde{i}_{J, I}(NT(\hat{I}_{W_J, S_J}))) \subseteq \span(\hat{\pi}(\mathfrak {N}))$. let  $x_1, x_2, \dots , x_n$ be a subset of 
$\hat{\pi}(\mathfrak {N})$ such that $\sum_{i = 1}^n\alpha_ix_i = 0$ for some $\alpha_1, \alpha_2, \dots , \alpha_n \in 
\mathbb{Q}$, then we have $x_i = \tilde{i}_{J, I}(p_i + \hat{I}_{W_J, S_J})$ for all $i$, for some finite $J \subseteq I$ with each   $p_i \in 
NT(\hat{I}_{W_J, S_J})$. 
Thus, since $\tilde{i}_{J, I}$ is a monomorphism by Lemma \ref{dirlim} (5), we have $\sum_{i = 1}^n\alpha_ip_i +  \hat{I}_{W_J, S_J} = 0$ in $\hat{R}_{W_{J}, S_{J}}$ and since the cosets corresponding to $NT(\hat{I}_{W_J, S_J})$ form a vector space basis in 
$\hat{R}_{W_{J}, S_{J}}$, we see that $\alpha_i = 0$ for $1 \leq i \leq n$. Hence the elements of $\hat{\pi}(\mathfrak {N})$ 
span $\hat{R}_{W, S}$ and are linearly independent over $\mathbb{Q}$ and thus form a basis for $\hat{R}_{W, S}$
as a vector space over $\mathbb{Q}$. 
\end{proof}

\begin{lemma}\label{Rtohat} Let $(W, S)$ be a Coxeter system and $\Gamma_{W, S}$, the associated graph. The inclusion map $i : P_{W, S} \to \hat{P}_{W, S}$ which identifies  the  basis of paths  $\mathfrak {P}_{W, S}$ in both algebras,  induces a monomorphism $\tilde{i} : R_{W, S} \to \hat{R}_{W, S}$.
\end{lemma}
\begin{proof} It is clear that $\tilde{i}$ is a well defined homomorphism of rings,
since $\mathfrak {C}_{W, S}$  is a generating set for both $I$ and $\hat{I}$.  
We let  $\pi$ and $\hat{\pi}$ denote the quotient 
maps from ${P}_{W,S}$ to $R_{W,S}$ and $\hat{P}_{W,S}$ to $\hat{R}_{W,S}$ respectively.
We let $\mathfrak {N} = \{x \in \mathfrak {P}_{W, S} | x \not= ptq, t \in T(\mathfrak {C}_{W, S}), p, q \in \mathfrak {P}_{W, S}\}$. 
By Lemma \ref{baaisN}, we have that $\hat{\pi}(\mathfrak {N})$ is a basis of  $\hat{R}_{W,S}$ as a vector space over $\mathbb{Q}$. 
Letting  $M$  denote the $\mathbb{Z}$-submodule of ${P}_{W, S}$ generated by $\mathfrak {N}$, and $V = \span(\mathfrak {N})$ the subspace of 
$\hat{P}_{W,S}$ as a $\mathbb{Q}$ vector space,
we get a commutative diagram;
 \begin{equation}\label{dia}
\xymatrix{
{P_{W,S}}\ar@{->}[rrrrr]^{{i}}\ar@{->}[ddr]_{\pi}&&&&&{\hat{P}_{W,S}}\ar@{->}[ddll]^{\hat{\pi}}\\
&{M}\ar@{_{(}->}[ul]\ar@{->}[rr]^(.3){i|_M}\ar@{->}[d]^{\pi|_M}&&{V = \span(\mathfrak {N})}\ar@{^{(}->}[urr]\ar@{->}[d]^{\hat{\pi}|_V}\\
&{R_{W, S}}\ar@{->}[rr]^{\tilde{i}}&&{\hat{R}_{W,S}}\\
}
\end{equation}
where
${\pi|_M}$ and 
$\hat{\pi}|_V$
denote the restrictions of the ring homomorphisms $\pi$ and $\hat{\pi}$ 
 as $\mathbb{Z}$ module and vector space homomorphism respectively. 
By Lemma  \ref{baaisN}, $\hat{\pi}|_V$ is an isomorphism of vector spaces. In particular $\hat{\pi}|_V$ is a monomorphism.

We claim that ${\pi|_M}$ maps $M$ onto $R_{W, S}$. We show that $P_{W, S} = M + I_{W, S}$.
First, we note that the paths of length zero, $\{[s_i] | s_i \in S \}$ are in $\mathfrak {N} \subset M$.  Next, we show that given any $p \in \mathfrak {N}$,  and any $y_{ij} \in Y$, with $j = o(p)$, then 
  $y_{ij}p  = q + I_{W, S}$ for some $q \in M$. If  $y_{ij}p \in \mathfrak {N}$, 
then $q = y_{ij}p$ and we are done. Otherwise, $p \in \mathfrak {N}$ and $y_{ij}p \not\in \mathfrak {N}$. Thus  $p$  is not divisible by any element of $T(\mathfrak {C}_{W, S})$
and  $y_{ij}p$ is divisible by some element of $T(\mathfrak {C}_{W, S})$. Comparing edges along the paths $p$ and $y_{ij}p$, we see that 
$y_{ij}p = T(C_{y_{ij}})p_1$, for some $p_1 \in \mathfrak {P}$. Thus, since  $C_{y_{ij}}$ is a monic polynomial over $\mathbb{Z}$  in $y_{ij}\bar{y}_{ij}$, we have $y_{ij}p - C_{y_{ij}}p_1$ is a $\mathbb{Z}$- linear combination of paths which divide $p$. 
Since paths which divide $p$ are already in normal form and are therefore in $\mathfrak {N}$,
 we have   $y_{ij}p - C_{y_{ij}}p_1\in M$ and  $y_{ij}p \in M + I_{W, S}$ as asserted. 
To prove our claim, we note that because  $[s_i] \in \mathfrak {N}$ for all  $s_i \in S$, and 
$y_{ij}p \in M  + I_{W, S}$ for all $y_{ij} \in Y$ and $p \in \mathfrak {N}$ such that $o(p) = j$, we have $q +  I_{W, S} \in M + I_{W, S}$ for any $q \in \mathfrak {P}$. Therefore as a $\mathbb{Z}$ module, 
$P_{W, S} = M + I_{W, S}$ and thus  ${\pi|_M}$ maps $M$ onto $R_{W, S}$,
 proving  our claim. 

Now since $\hat{P}_{W, S} \cong \mathbb{Q}\otimes_{\mathbb{Z}}P_{W, S}$ and $P_{W, S}$ is a free $\mathbb{Z}$ module, 
we have that $i$ is a monomorphism. Thus its restriction $i|_M$ is also  a monomorphism. Because $\pi|_M$ is an epimorphism and $i|_M$ and $\hat{\pi}|_V$ are both monomorphisms, we can conclude from our diagram $\ref{dia}$ that  $\tilde{i}$ is a monomorphism as desired.
\end{proof}

\begin{corollary}\label{corN}Let $(W, S)$ be a Coxeter system and $\Gamma_{W, S}$, the associated graph. 
 Let   ${\pi}$ denote the quotient 
map ${P}_{W,S}$ to ${R}_{W,S}$. Let $\mathfrak {N} = \{x \in \mathfrak {P}_{W, S} | x \not= ptq, t \in T(\mathfrak {C}_{W, S}), p, q \in \mathfrak {P}_{W, S}\}$. 
The 
elements of the set ${\pi}(\mathfrak {N})$ are pairwise distinct and  form a basis for ${R}_{W, S}$ as a  $\mathbb{Z}$-module.
\end{corollary}

 \begin{proof} As in the proof of Lemma \ref{Rtohat}, let $M$ be the $\mathbb{Z}$-submodule of ${P}_{W,S}$ spanned by $\mathfrak {N}$. 
 Since $\pi|_M$ is onto, we get that $\pi(\mathfrak {N})$ spans  ${R}_{W, S}$ as a  $\mathbb{Z}$-module. From the proof of Lemma \ref{Rtohat},
 we have a commutative diagram:
  \begin{equation}\label{dia2}
\xymatrix{
{P_{W,S}}\ar@{->}[rrrrr]^{{i}}\ar@{->}[ddr]_{\pi}&&&&&{\hat{P}_{W,S}}\ar@{->}[ddll]^{\hat{\pi}}\\
&{M}\ar@{_{(}->}[ul]\ar@{->}[rr]^(.3){i|_M}\ar@{->}[d]^{\pi|_M}&&{V = \span(\mathfrak {N})}\ar@{^{(}->}[urr]\ar@{->}[d]^{\hat{\pi}|_V}\\
&{R_{W, S}}\ar@{->}[rr]^{\tilde{i}}&&{\hat{R}_{W,S}}\\
}
\end{equation}
 where $\hat{\pi}|_V\circ i|_M$ is a monomorphism. Thus $\tilde{i}\circ \pi|_M$ is a monomorphism and since $\tilde{i}\circ \pi|_M(\mathfrak {N})$ forms a 
 basis over $\mathbb{Q}$ for $\hat{R}_{W, S}$ by Lemma \ref{baaisN}, we must have that ${\pi}(\mathfrak {N})$ are linearly independent over $\mathbb{Z}$ in  ${R}_{W, S}$.
 \end{proof}
 
  \begin{exmp} 
   Let $(W, S)$ be the Coxeter system with Coxeter matrix 
   $$\left(\begin{array}{ccccc}1 & 3 & 2 & 4 & 2 \\3 & 1 & 5 & 2 & 2 \\2 & 5 & 1 & 6 & 5 \\4 & 2 & 6 & 1 & \infty \\2 & 2 & 5 & \infty & 1\end{array}\right),$$
  then $(W, S)$ be a Coxeter system with generators   $S = \{s_1 = r, s_2 = s, s_3 = t, s_4 = u, s_5 = v\}$ 
  and relations 
  $$r^2 = s^2 = t^2 = u^2 = v^2 = 1 =(rs)^3 = (st)^5 = (tu)^6 = (tv)^5 = (ru)^4 = 
  (rt)^2 = (su)^2 = (vr)^2 = (vs)^2.$$
   Let 
  $\Gamma_{W, S}$ be the associated graph. On the left below, we show the edges in $Y_{W, S, +}$ as before 
   on the right,  we label each geometric edge $\{s_is_j\},
 s_i, s_j \in S$, in the corresponding graph with  $m_{ij}$, if $(s_is_j)^{m_{ij}} = 1$ is among the above relations and $m_{ij} \not= 2$. The remaining edges are labelled with infinity. 
 
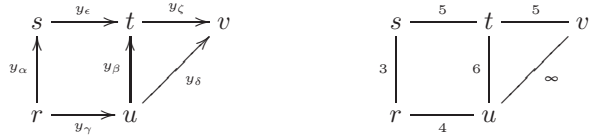
\begin{equation*}
\xymatrix{s\ar@{->}[r]^{y_{\epsilon}}&t\ar@{->}[r]^{y_\zeta}&{v}&\\
{r}\ar@{->}[u]^{y_\alpha}\ar@{->}[r]_{y_\gamma}&{u}\ar@{->}[u]^{y_\beta}\ar@{->}[ur]_{y_\delta}\\
}
\qquad
\xymatrix{
s\ar@{-}[r]^{5}&t\ar@{-}[r]^{5}&{v}&\\
{r}\ar@{-}[u]^{3}\ar@{-}[r]_{4}&{u}\ar@{-}[u]^{6}\ar@{-}[ur]_{\infty}\\
}
\end{equation*}
We have  $\hat{P}_{W, S} = \mathbb{Q}\Gamma_{W, S}$ and $\hat{I}_{W, S}$ is  the ideal of $\hat{P}_{W, S}$ with the following  generators:

$$
\begin{array}{lll}
C_{y_\alpha} = [rsr] - [r],  & C_{{y}_\gamma} = [rur] - 2[r],  & C_{{y}_\beta} = [utu] - 3[u], \\[1em]
 C_{\bar{y}_\alpha} = [srs]-[s],  &C_{\bar{y}_\gamma} = [uru] - 2[u], & C_{\bar{y}_\beta} = [tut] - 3[t]\\[1em]
C_{{y}_{\epsilon}} = [ststs] - 3[sts] + [s], &C_{{y}_\zeta}  = [tvtvt] - 3[tvt] + [t], &\\[1em]
 C_{\bar{y}_{\epsilon}} = [tstst] -3[tst] +[t], & C_{\bar{y}_\zeta} = [vtvtv] -3[vtv] +[v].  & \\[1em]
\end{array}
$$ 
 \end{exmp}

 \subsection{The Ring  $\widetilde R$}
 In this section, we give the definition of the ring  $\widetilde R =  \widetilde R_{W, S}$ for a Coxeter System. The ring has relations given by polynomials satisfying a recurrence relation. To facilitate a discussion of the representation theory motivating the theory in the next section, we give the definition of the 
 polynomials in the more general setting of a ring with idempotents. 
 \subsection{} \label{cheb} Let $B$ be a (possibly non-unital) ring. For any idempotents $e,f\in B$ and elements $a\in fBe$, $b\in eBf$, define elements $c_{n}(a,b,e,f)\in B$ for $n\in \mathbb{Z}$  by the recurrence formulae
\[c_{2n+2}(a,b,e,f)=ac_{2n+1}(a,b,e,f)-c_{2n}(a,b,e,f),\qquad c_{2n+1}(a,b,e,f)=bc_{2n}(a,b,e,f)-c_{2n-1}(a,b,e,f)\]
for $n\in \mathbb{Z}$  and initial conditions
\[c_{0}(a,b,e,f)=0,\qquad c_{1}(a,b,e,f)=e.\]
We have 
\[c_2(a,b,e,f) = a, \ c_3(a,b,e,f) = ba - e, \ c_4(a,b,e,f) = aba - 2a, \ c_5(a,b,e,f) = baba - 3ba  + e, \dots \]

Define also elements $C_{n}(a,b,e,f)\in B$ for $n\in \mathbb{N}_{\geq 2}$ as follows. Let $C_{2}(a,b,e,f)=a$.
For $n\geq 3$, let 
$C_{n}(a,b,e,f):=\rho_{B,e,a,b}(C_{n}(t))$  where $C_{n}(t) \in \mathbb{Z}[t]$ is the minimal polynomial over $\mathbb{Q}$ of (the algebraic integer) $\displaystyle 4\cos^{2}\frac{\pi}{n}$, and  $\rho_{B,e,a,b}$ denotes the the ring homomorphism $\rho_{B,e,a,b}\colon\mathbb{Z}[t]\to eBe$ which maps $1\mapsto e$ and $t\mapsto ba$ (note that the ring $eBe\ni e,ba$ has identity element $e$).
We see that 
\[C_3(a,b,e,f) = ba - e, \ C_4(a,b,e,f) = ba - 2e,  \ C_5(a,b,e,f) =  baba - 3ba  + e, \ C_6(a,b,e,f) = ba - 3e, \ \dots \]
We write $c_{n}=c_{n,B}$ and  $C_{n}=C_{n,B}$ if it is necessary to indicate dependence on the ring $B$. Then for a ring homomorpism $\theta\colon B\to B'$, one has
$\theta(c_{n,B}(a,b,e,f))=c_{n,B'}(\theta(a),\theta(b),\theta(e),\theta(f))$ and similarly with $c_{n}$ replaced by $C_{n}$ if $n\geq 2$.

\subsection{} \label{fact} By  \cite{Noncom},    $c_{n}(a,b,e,f)=\prod_{N\in\mathbb{N}_{\geq 2},N\mid n}C_{N}(a,b,e,f)$ for all $n\in \mathbb{N}_{\geq 2}$, where $N\mid n$ means $N$ is a divisor of $n$ in $\mathbb{Z}$ and  the order of factors in the product on the right is immaterial except that $C_{2}(a,b,e,f)$ must be the leftmost factor if $n$ is even. In fact, if the pairwise distinct  positive, non-unit integer divisors, of $N$ are  $ N_{1},\ldots, N_{k}$, then  
 $c_{n}(a,b,e,f)=C'_{N_{1}}\ldots C'_{N_{k}}$ where 
 $C'_{N}:=C_{N}(a,b,e,f)$ unless $n$ is even and  $N=N_{i}$
 where $1\leq i < j\leq k$ and $N_{j}=2$, in which case
 $C'_{N}:=C_{N}(b,a,f,e)$. The equivalence of these various factorizations is trivial on noting that all $C_{n}(a,b,e,f)$ with $n\geq 3$ lie in the commutative unital subring of $B$ generated by $e$ and $ba$, with $aC_{n}(a,b,e,f)=C_{n}(b,a,f,e)a$.

 \subsection{}\label{noncomrep} We now take notation as in \ref{cox}. Thus, $P_{W,S}$ is the path algebra over $\mathbb{Z}$  of $\Gamma_{W,S}$, $I_{W,S}$ is an  ideal  of $P_{W,S}$ and $R_{W,S}$ is the quotient ring $R_{W,S}:=P_{W,S}/I_{W,S}$.

By definition,  $I_{W,S}$ is the two-sided ideal of $P_{W,S}$ generated by  elements
$C_{[sr]} = C_{m_{r, s}}([srs]) = C_{m_{r,s}}([rs],[sr],[s],[r])$ of \ref{cheb}, for $(r,s)\in S\times S$ such that
$2<m_{r,s}<\infty$. Let $\widetilde{I}_{W,S}$ denote the two-sided ideal of $P_{W,S}$ generated by the elements
$c_{m_{r,s}}([rs],[sr],[s],[r])$ for $(r,s)\in S\times S$ such that
$2<m_{r,s}<\infty$. From \ref{cheb}, we have $\widetilde{I}_{W,S}\subseteq I_{W,S}$. Define the quotient ring
$\widetilde R_{W,S}:=P_{W,S}/\widetilde{I}_{W,S}$ of $P_{W,S}$.
The above gives a canonical surjective  ring homomorphism $\widetilde R_{W,S}\to R_{W,S}$.

\subsection{Parameterizing Bases of $R_{W,S}$ and $\hat{R}_{W,S}$.}\label{refreps}
 Recall from Tits solution to the word problem for $(W,S)$  that   a sequence $(s_{1},\ldots, s_{n})$ in $S$ is the unique reduced expression of the corresponding product $s_{1}\cdots s_{n}\in W$ if and only if  it contains no consecutive subsequence $(s_{i},\ldots,s_{j})$, where $1\leq i<j\leq n$, of length $j-i+1$ which  either  is of the form $(s,s)$ for some $s\in S$ or is an alternating sequence $(s,r,s,\ldots)$ where $r,s\in S$ are distinct with $m_{r,s}=j-i+1$.  Let $W_{1}$ be the set of all such sequences with $n\geq 1$.

Let $\varphi\colon \mathbb{N}_{\geq 1}\to \mathbb{N}$ denote the Euler totient function (that is, $\varphi(n)=\vert (\mathbb{Z}/n\mathbb{Z})^{*}\vert$ where  $(\mathbb{Z}/n\mathbb{Z})^{*}$ indicates the unit group. Note that  if $n\geq 3$, then  $\frac{\varphi(n)}{2}$ is the degree of $C_{n}(t)\in \mathbb{Z}[t]$.  

Let $W_{1}'$ be the set of all sequences $(s_{1},\ldots, s_{n})$ in $S$, with $n\geq 1$,  which contain no  consecutive 
subsequence $(s_{i},\ldots,s_{j})$ which is of the form
$(s,r,s,\ldots)$ where $r,s\in S$, $ m_{r,s}<\infty$ and $j-i+1
=1+\varphi(m_{r,s})$.  Note that $W'_{1}\subseteq W_{1}$.

It is known from \cite{Noncom} that the elements    $[s_{1}\ldots s_{n}]$ of $\widetilde R_{W,S}$ for $(s_{1},\ldots, s_{n})\in W_{1}$ are pairwise distinct   and form a $\mathbb{Z}$-module basis of $\widetilde R_{W,S}$\footnote{A different basis for $\widetilde R$, with some very favorable properties, is also considered in \cite{Noncom}}. 

\begin{theorem}\label{basispar}
\begin{enumerate}
\item The  elements $[s_{1}\ldots s_{n}]$ of $ \hat R_{W,S}$ for $(s_{1},\ldots, s_{n})\in W'_{1}$ are pairwise distinct   and form a basis of $\hat R_{W,S}$ as vector space over $\mathbb{Q}$. 
\item The  elements $[s_{1}\ldots s_{n}]$ of $R_{W,S}$ for $(s_{1},\ldots, s_{n})\in W'_{1}$ are pairwise distinct   and form a $\mathbb{Z}$-module basis of $R_{W,S}$.
\end{enumerate} \end{theorem}
\begin{proof} Part (1) follows from Lemmas \ref{baaisN} and Part (2) follows from Corollary \ref{corN}.
 \end{proof}
 
  \begin{exmp}\label{Weyl}  (1) Suppose that $m_{r,s}\in \{2,3,4,6\}$ for all $r\neq s$ in $S$.  Since $\varphi(m)=2$ for $m=3,4,6$,     the sequences $(s_{1},\ldots, s_{n})\in W_{1}'$ are precisely the  sequences of successive vertices of non-backtracking paths in $\Gamma_{W,S}$. If, further,
 the Coxeter graph of $(W,S)$ is a tree, then for any  $r,s\in S$, there is a unique non-backtracking path in $\Gamma_{W,S}$ from $r$ to $s$ and it follows that
 \[ \dim_{\mathbb{Q}}([r] \hat R_{W,S} [s])=\mathrm{rank}_{\mathbb{Z}}([r] R_{W,S} [s])=1,\] and so  
  \[ \dim_{\mathbb{Q}}( \hat R_{W,S})=\mathrm{rank}_{\mathbb{Z}}(R_{W,S} )=\vert S\vert^{2}.\]  In particular, these facts all hold  if $(W,S)$ is an 
  irreducible Weyl group or an irreducible affine Weyl group which is not of type $\widetilde A_{n}$ for any $n\in \mathbb{N}_{\geq 1}$. 
  
  (2) Suppose that for each $r\neq s$ in $S$, $m_{r,s}$ is either $\infty$ or a prime integer. Then $I_{W,S}=\widetilde I_{W,S}$ and  hence $R_{W,S} =\widetilde R_{W,S}$.
 \end{exmp}

 \section{Reflection Representations}\label{secrealreps}
 In this section, we describe results of Dyer on the relation between standard reflection representations of Coxeter systems in real vector spaces and lax and strict reflection representations over non-commutative rings. 
 \subsection{}  
 Let $A$ be a commutative ring. Let $B$ be an $A$-algebra, and $M$ a left $A$-algebra module for $B$, as defined in Section \ref{intro}. 
For a group $G$,  a representation of $G$ on the $B$-module  $M$ is by definition a group homomorphism $\theta\colon G\to \mathrm{Aut}_{B}(M)$. We usually  write $gm:=(\theta(g))(m)$ and say simply that $G$ acts $B$-linearly on $M$ (leaving tacit the condition that the action is also $A$-linear). Similar conventions apply to right $G$ actions and right modules.

 For the $A$-algebra $B$, the   left (resp., right) regular $B$-module  is $B$ as $A$-module,   with $B$-action by left (resp., right)  multiplication. 
If  $B$ is unital and $M$ is its left regular module, then $\mathrm{End}_{B}(M)\cong B^{\mathrm{op}}$, where $B^{\mathrm{op}}$ is the opposite  algebra,  and $\mathrm{Aut}_{B}(M)$ identifies with the group
 of  units  of $B^{\mathrm{op}}$. If $B$ is commutative  unital, we often write  $\mathrm{GL}_{B}(M):= \mathrm{Aut}_{B}(M)$.

\subsection{}\label{refdatum}  Let $M$  be a left $B$-module, $\ck M$ be a right $B$-module and $\mpair{-,-}\colon M\times \ck M
\to B$ be a $B$-bilinear map. That is, for all $r,r'\in B$,  we require 
$\mpair{rm+r'm',m''}=r\mpair{m,m''}+r'\mpair{m',m''}$ for $m,m'\in M$ and  $m''\in \ck M$, and
$\mpair{m'', mr+m'r'}=\mpair{m'',m}r+\mpair{m'',m'}r'$ for $m''\in M$ and $m,m'\in \ck M$. We also require these conditions to hold with $r$ and $r'$ in the coefficient ring, $A$,  of $B$.  Suppose given families
$(\alpha_{s})_{s\in S}$ in $M$, $(\ck\alpha_{s})_{s\in S}$ in $\ck M$ and $(e_{s})_{s\in S}$ in $B$ such that for all $s\in S$  $e_{s}^{2}=e_{s}$, $e_{s}\alpha_{s}=\alpha_{s}$, $ \ck \alpha_{s}e_{s}=\ck\alpha_{s}$ and $\mpair{\alpha_{s},\ck\alpha_{s}}=2e_{s}$.
Define the $S\times S$ matrix $A=(a_{r,s})_{r,s\in S}$ with entries in $B$ by $a_{r,s}:=\mpair{\alpha_{r},\ck\alpha_{s}}\in e_{r }B e_{s}$.

Define $B$-module  endomorphisms $\phi_{s}\in \mathrm{End}_{B}(M)$ of $M$ for $s\in S$   by $\phi_{s}(m)=m-\mpair{m,\ck\alpha_{s}}\alpha_{s}$ for all $m\in M$.
It is easy to check that $\phi_{s}^{2}=\mathrm{Id}_{M}$ for all $s\in S$, so $\phi_{s}\in \mathrm{Aut}_{B}(M)$.

The following is a main result in Dyer \cite{Noncom}:

\begin{prop}\label{refrep}(Dyer \cite{Noncom}) Suppose notation is as above and that for all $(r,s)\in S\times S$ with $r\neq s$ and $m_{r,s}\neq \infty$, one has
$c_{m_{r,s}}(a_{s,r},a_{r,s},e_{r},e_{s})=0$.
Then there is a left $B$-linear left $W$-action on $M$ 
determined by $sm=m-\mpair{m,\ck\alpha_{s}}\alpha_{s}$ for all $m\in M$ and $s\in S$ and a right $B$-linear  left $W$-action    on $\ck M$ determined by $sm'=m'-\ck\alpha_{s}\mpair{\alpha_{s},m'}$ for all $m'\in \ck M$. These satisfy $\mpair{wm,wm'}=\mpair{m,m'}$ for all $m\in M$, $m'\in \ck M$ and $w\in W$.
\end{prop}
\subsection{}  We call the matrix $(a_{r,s})_{r,s\in S}$  a non-commutative generalized Cartan matrix (NCM) when the assumptions in \ref{refrep} hold. In that case, 
 the left $W$-action on $M$ gives a representation  $\phi\colon W\to \mathrm{Aut}_{B}(M)$,
and similarly for $\ck M$.  We call these lax reflection representations of $(W,S)$ on $M$ and $\ck M$.  If, more strongly, one has 
$C_{m_{r,s}}(a_{s,r},a_{r,s},e_{r},e_{s})=0$ for all $(r,s)\in S\times S$ with $r\neq s$ and $m_{r,s}\neq \infty$, we say these representations are strict reflection representations and say that $A$ is a strict NCM.

In either case, the subset  $\Phi:=\{w\alpha_{s}\mid w\in W,s\in S\}$ of $M$ is called the (corresponding)   root system of $(W,S)$ in $M$, and  $\ck\Phi:=\{w\ck\alpha_{s}\mid w\in W,s\in S\}$ is called the coroot system of $(W,S)$ in $\ck M$.  For some basic properties of  lax reflection representations and their root systems, see \cite{Noncom}.
 
\begin{exmp}\label{realrefrep} Suppose above that 
$B=\mathbb{R}$,
$M=V$ and $\ck M=\ck V$ are $\mathbb{R}$-vector spaces 
(unital $B$-modules) and $e_{r}=1_{B}$ for all $r\in S$. Let 
$a_{r,s}:=\mpair{\alpha_{r},\ck\alpha_{s}}$.  The proposition 
 produces lax reflection  representations 
$ W\to \mathrm{GL}_{\mathbb{R}}(V)$ and 
$ W\to \mathrm{GL}_{\mathbb{R}}(\ck V)$   provided that  $a_{s,s}=2$ for all 
$s\in S$, and, for all $r\neq s\in S$,   $a_{r,s}=0$ if $m_{r,s}=2$
 and   $a_{r,s}a_{s,r}=4\cos^{2}\frac{k_{r,s}\pi}{m_{r,s}}$ for 
 some $1\leq k_{r,s}\leq  m_{r,s}/2$ if $3\leq m_{r,s}< \infty$. 
 These reflection representations are strict if and only if   
 $k_{r,s}$ is relatively prime to $m_{r,s}$ whenever 
 $3\leq m_{r,s}< \infty$. We call the matrix $A:=(a_{r,s})_{r,s}$ a strict (resp., lax) real reflection matrix, abbreviated SRRM (resp., LRRM)  for $(W,S)$ if it arises in ths way from a strict (resp., lax) reflection representation.  \end{exmp}
 
\begin{rem*} (1) Representations as in the preceding example are mostly studied under    conditions which ensure they are strict and that the representation and its  root system have standard (convexity and positivity) properties.  Essentially the most general  such  conditions   are the following:

\begin{enumerate}
\item $a_{r,r}=2$ for all $r\in S$.
\item $a_{r,s}\leq 0$ for all $r\neq s$ in $S$.
\item $a_{r,s}=0$ if  $r,s\in S$ and $m_{r,s}=2$.
\item $a_{r,s}a_{s,r}=4\cos^{2}\frac{\pi}{m_{r,s}}$ if $r,s\in S$ and $3\leq m_{r,s}<\infty$.
\item $a_{r,s}a_{s,r}\geq 4$ if $r,s\in S$ and $ m_{r,s}=\infty$.
\item $\Pi$ and $\ck \Pi$ are positively independent.
\end{enumerate}
Here, a family $(\beta_{i})_{i\in I}$ of elements in a real vector space is said to be positively independent if $\sum_{i\in I} a_{i}\beta_{i}=0$ for  non-negative real scalars $a_{i}$, almost all zero, implies $a_{i}=0$ for all $i\in I$.  In particular, linearly independent vectors are positively independent. When (1)--(5) hold, we say that $A:=(a_{r,s})_{r,s\in S}$ is a (possibly) non-crystallographic 
generalized Cartan matrix (NGCM) for $(W,S)$. The finite NGCMs with integral entries are essentially the Generalized Cartan Matrices (GCMs) appearing in the study of Kac-Moody Lie algebras \cite{Kac} and elsewhere.  Note that NGCMs are SRRMs.  See    \cite{Paired}, \cite{FuNonortho}, \cite{Conjugacy} and \cite{Noncom} for   further details on reflection representations  attached to NGCMs. 

(2) In some very special  cases including  finite irreducible Coxeter systems, real reflection representations attached to
SRRMs may be viewed as ``Galois twists''  of those attached to NGCMs, but this is not the case in general (a precise discussion of this would  involve reflection representations defined over  subfields of $\mathbb{R}$). 

(3) For any Coxeter system $(W',S)$ for which there is a surjective group homomorphism $W\to W'$ which is the identity on $S$, a SRRM for $(W',S')$ is a 
LRRM for $(W,S)$.

\end{rem*}

The following proposition summarizes   part of the above discussion which is used  subsequently. 
\begin{prop}\label{NGCMrepsum} Let $A=(a_{r,s})_{r,s\in S}$ be a LRRM  for $(W,S)$ and $V=V_{A}$ , $\ck V=\ck V_{A}$ be  real vector spaces with bases $\{\alpha_{s}\mid s\in S\}$ and  $\{\ck\alpha_{s}\mid s\in S\}$ respectively. Then there are  $\mathbb{R}$-linear lax  reflection representations of  $W$   on $V$ and $\ck V$   such that for $s,r\in S$, one has 
$   s\alpha_{r}=\alpha_{r}-a_{r,s}\alpha_{s}$ and 
$   s\ck\alpha_{r}=\ck\alpha_{r}-a_{s,r}\ck\alpha_{s}$.
We call  $(V_{A},\ck V_{A})$  the standard   paired (lax) real reflection representations of $(W,S)$  with LRRM  $A$.
If $A$ is a SRRM, these representations are strict. \end{prop}
\subsection{} 
 Let $A$ be a LRRM for $(W,S)$ and let  $(V,\ck V)=(V_{A},\ck V_{A})$ be  the   standard paired real reflection representations of $(W,S)$  with 
 LRRM $A$. As constructed above, $V$ and $\ck V$ have left 
 $W$-actions, but we shall find it convenient to regard $W$ as 
 instead acting on the right of $V$, by defining $vw:=w^{-1}v$ for 
 $v\in V$ and $w\in W$. We still regard the $W$-action on 
 $\ck V$ as a left action.
 Then $\ck V\otimes_{\mathbb{R}} V$  has natural commuting left and right $W$-actions, defined by
 \[ w(v'\otimes v)=(wv')\otimes v,\qquad  (v'\otimes v)w=v'\otimes (vw), \qquad \text{\rm if $v'\in \ck V$, $v\in V$, $w\in W$.} \] 
 One has
 \[ s(\ck\alpha_{r}\otimes \alpha_{t})=
 \ck\alpha_{r}\otimes \alpha_{t}-a_{s,r}\ck\alpha_{s}\otimes\alpha_{t},\qquad  (\ck\alpha_{r}\otimes \alpha_{t})s=
 \ck\alpha_{r}\otimes \alpha_{t}-a_{t,s}\ck\alpha_{r}\otimes\alpha_{s},\qquad \text{\rm for $r,s,t\in S$.}\]

\subsection{}   Let $B:=\mathrm{Mat}'_{S\times S}(\mathbb{R})$ denote the $\mathbb{R}$-algebra  of real $S\times S$-matrices
with only finitely many non-zero entries. Note that $B$ has a $\mathbb{R}$-basis $(e_{r,s})_{r,s\in S}$ where  $e_{r,s}$ is the matrix unit (the matrix with $(r',s')$-entry.$\delta_{r,r'}\delta_{s,s'}$ where $\delta_{j,k}$ denotes the Kronecker delta in $\mathbb{R}$). 

 We shall also regard $B$ as a $(B,B)$-bimodule in the natural way.  Identify
$\ck V\otimes_{\mathbb{R}}V=B$ as $\mathbb{R}$-vector space so that $\ck \alpha_{r}\otimes \alpha_{s}=e_{r,s}$. This endows $B$ with structure of $\mathbb{R}$-algebra and with  left and right $W$-actions which  commute, such that the left (resp., right) $W$-action is right (resp., left) $B$-linear.

\subsection{} Consider any ring  ${R'}:=P_{W,S}/I$ where $I$ is a two-sided ideal of $P_{W,S}$ such that $\widetilde{I}_{W,S}\subseteq I$. The two cases of greatest interest are  (1) $I=I_{W,S}$ and ${R'}=R_{W,S}$, and (2)  $I=\widetilde{I}_{W,S}$ and ${R'}=\widetilde R_{W,S}$. For any element $x$ of 
$P_{W,S}$, we denote its image in ${R'}$ under the canonical surjection $P_{W,S}\to {R'}$ still by $x$. In particular, this defines
$[x_{1}\ldots x_{n}]\in {R'}$ if $n>0$ and  $x_{1},\ldots, x_{n}\in S$ are the successive vertices of a path in $\Gamma_{W,S}$. We also set
$[x_{1}\ldots x_{n}]:=0_{{R'}}$ if $n>0$ and  $x_{1},\ldots, x_{n}\in S$ are not the successive vertices of a path in $\Gamma_{W,S}$ (that is, if for some $1\leq i\leq n-1$,  $m_{x_{i},x_{i+1}}\leq 2$). Note that the elements $[s]$ for $s\in S$ are pairwise orthogonal idempotents in ${R'}$ such that ${R'}=\bigoplus_{r,s\in S}[r]{R'}[s]$ as abelian group.
\begin{prop} Define the quotient ring ${R'}=P_{W,S}/I$ where $I\supseteq \widetilde I_{W,S}$ is a two-sided ideal of $P_{W,S}$. Let $M$ (resp., $\ck M$) denote the left  regular ${R'}$-module ${R'}$. 
\begin{enumerate}\item There is  a (unique)
left $R'$-linear lax refelction representation of $W$ on $M$ such that  
\[   s[r]=\begin{cases} - [s],&\text{\rm if $r=s$}\\ 
  [r]+[rs], &\text{\rm if $r\neq s$.}
\end{cases} \qquad \text{\rm for $r,s\in S$.}\] and a (unique) right 
$R'$-linear lax reflection representation of $W$ on $\ck M$ such that 
\[  s [r]=\begin{cases} - [s],&\text{\rm if $r=s$}\\ 
  [r]+[sr], &\text{\rm if $r\neq s$.}
\end{cases}\qquad \text{\rm for $r,s\in S$.}\]
\item Base change $-\otimes_{\mathbb{Z}}\mathbb{R}$ gives a (possibly non-unital)  $\mathbb{R}$-algebra $R'\otimes_{\mathbb{Z}}\mathbb{R}$ and left (resp., right)
$R'\otimes_{\mathbb{Z}}\mathbb{R}$-linear lax left reflection actions of $W$ on the left and right regular   $R'\otimes_{\mathbb{Z}}\mathbb{R}$-modules $M\otimes_{\mathbb{Z}}\mathbb{R}$ and 
$\ck M\otimes_{\mathbb{Z}}\mathbb{R}$ respectively.
\item If $I\supseteq \widetilde I_{W,S}$, then the lax reflection representations in (1)--(2) are strict.
\end{enumerate}  
\end{prop}
\begin{proof} We prove (1).
 Define families $(e_{s})_{s\in S}$ in ${R'}$, $(\alpha_{s})_{s\in S}$ in $M$ and $(\ck\alpha_{s})_{s\in S}$ in $\ck M$ by setting $\alpha_{s}=\ck\alpha_{s}:=[s]$ for all $s\in S$.   A ${R'}$-bilinear map $\mpair{-,-}\colon M\times \ck M\to {R'}$ is uniquely determined by the values  $\mpair{[r],[s]}$ for $r,s\in S$, which may be arbitrarily assigned subject only to the conditions that 
$\mpair{[r],[s]}\in e_{r}{R'}e_{s}$ for all $r,s$. Hence there is a
unique ${R'}$-bilinear map $\mpair{-,-}\colon M\times \ck M\to {R'}$
such that \begin{equation*}
\mpair{\alpha_{r},\ck\alpha_{s}}=\begin{cases}2[r],&\text{\rm if $r=s$}\\
-[rs],&\text{\rm if $r\neq s$.}
\end{cases}
\end{equation*}
Note that $\mpair{\alpha_{r},\ck\alpha_{s}} = 0$ if $m_{rs} = 2$ by our conventions. 
The conditions of Proposition \ref{refrep} are satisfied, by definition of $\widetilde{I}_{W,S}$ and the assumption that $\widetilde{I}_{W,S}\subseteq I$. Hence  the proposition gives lax representations as required in (1). Part (2) follows easily by base change and (3) is immediate from the definitions. \end{proof}
\subsection{} Regard the left $R'$-linear $W$-action just defined
 on $M$ as a right action by setting $mw:=w^{-1}m$ for $w\in W$ and $m\in M$,  and  henceforward write $M$ and $\ck M$  just as $R'$. Then $R'$ admits a right $R'$-linear left reflection  action of $W$,  and a left $R'$-linear   right reflection action of $W$.
 Moreover, these actions commute, as one can see by direct computation or from the following lemma.
  \begin{lemma} Let $C$ be any idempotent ring (i.e. $C^{2}=C$).
 Let $\theta\colon C\to C$ (resp., $\tau\colon C\to C$) be an endomorphism of the left (right) regular $C$ module. Then $\tau$ and $\theta$ commute in $\mathrm{End}_{\mathbb{Z}}(C)$.  \end{lemma} 
 \begin{proof}
  Let $c\in C$. Since $c\in C=C^{2}$, we may write
  $c=\sum_{i=1}^{n}c_{i}d_{i}$ for some $n\in \mathbb{N}$ and $c_{i},d_{i}\in C$. Then
  \begin{equation*}
  \tau(\theta(c))
 =\tau(\theta(\sum_{i=1}^{n}c_{i}d_{i}))=
  \tau(\sum_{i=1}^{n}c_{i}\theta(d_{i}))=
   \sum_{i=1}^{n}\tau(c_{i})\theta(d_{i})=\theta( \sum_{i=1}^{n}\tau(c_{i})d_{i})=\theta(\tau(\sum_{i=1}^{n}c_{i}d_{i})=\theta(\tau(c)).\qedhere
  \end{equation*}  
 \end{proof}
 \subsection{} Let $A:=(a_{r,s})_{r,s\in S}$ be an $S\times S$-indexed family in $\mathbb{R}$.  Then there is a unique ring  homomorphism $\theta'_{A}\colon P_{W,S}\to B$
 such that for any path in $\Gamma_{W,S}$ with successive vertices $s_{1},\ldots, s_{n}$, one has 
$\theta'_{A}([s_{1}\cdots s_{n}])=(-1)^{n-1}a_{s_{1},s_{2}}\cdots a_{s_{n-1},s_{n}}e_{s_{1},s_{n}}$ (where here and below, $a_{s_{1},s_{2}}\cdots a_{s_{n-1},s_{n}}=1_{\mathbb{R}}$ by convention if $n=1$). Note that $\theta'_{A}$ doesn't depend on the values $a_{r,s}$ with $m_{r,s}\leq 2$. 

  \begin{prop} Let $A$ and $\theta'_{A}$ be as above. \label{homsm} 
\begin{enumerate}\item If $A$ is a LRRM for $(W,S)$,  then $\theta'_{A}$ factors as  $ {P}_{W,S}\twoheadrightarrow \widetilde  R_{W,S}\xrightarrow{\theta_{A,\widetilde  R}} B$ for a unique ring homomoprhism $\theta_{A,\widetilde  R}$.
\item If $A$ is a SRRM for $(W,S)$, then $\theta_{A,\widetilde{R}}$ factors as 
$\widetilde {R}_{W,S}\twoheadrightarrow R_{W,S}\xrightarrow{\theta_{A,R}} B$ for  a unique ring homomoprhism $\theta_{A, R}$.
\item Suppose $A$ is a LRRM. Let $(S_{i})_{i\in I}$ be the equivalence classes for the finest equivalence  relation $\sim$ on $S$ such that $r\sim s$ if $a_{r,s}\neq 0$. 
Then $\span(\mathrm{Img}(\theta_{A,\widetilde R}))=\bigoplus_{i\in I}
 \mathrm{Mat}'_{S_{i}\times S_{i}}(\mathbb{R})\subseteq B$. 
\item If  $A$ is a SRRM, then 
$\span(\mathrm{Img}(\theta_{A,{R'}}))= 
\bigoplus_{i\in I}\mathrm{Mat}'_{S_{i}\times S_{i}}(\mathbb{R})$ where $(W_{i},S_{i})_{i\in I}$ are the irreducible compoenents of the Coxeter system $(W,S)$.
\end{enumerate}\end{prop}
\begin{proof}  We prove (1). Let  $A$ be  a LRRM. Since $\widetilde R_{W,S} =P_{W,S}/  \widetilde  I_{W,S}$, it will suffice to show that $\ker \theta'_{A}\supseteq \widetilde  I_{W,S}$. Recall that $\widetilde  I_{W,S}$ is generated by elements $c_{m}([rs],[sr],[s],[r])\in P_{W,S}$ where   $r,s\in S$ with $3\leq m:=m_{r,s}<\infty$. 
Let $x:=e_{r,r}\in B$ if $m$ is odd and $x:=e_{s,r}\in B$
if $m$ is even. It is easily seen that
\[\begin{split}\theta'_{A,}(c_{m,P}([rs],[sr],[s],[r])&=c_{m,B}(-a_{r,s}e_{r,s},-a_{sr}e_{s,r},e_{s,s},e_{r,r})=c_{m,\mathbb{R}}(-a_{r,s},-a_{sr},1_{\mathbb{R}},{1}_{\mathbb{R}})x\end{split}.\]
Since $A$ is a LRRM, we may write 
 $a_{r,s}a_{s,r}=4\cos^{2}\frac{k\pi}{n}$ where $k\in \mathbb{N}$ and $n\in \mathbb{N}_{\geq 2}$ are relatively prime, $n\vert m$ and $1\leq k<m$. 
 We   have $C_{n,\mathbb{R}}(-a_{r,s},-a_{sr},1_{\mathbb{R}},{1}_{\mathbb{R}})=C_{n}\left(4\cos^{2}\frac{k\pi}{n}\right)=0$. So  $c_{m,-\mathbb{R}}(-a_{r,s},-a_{sr},1_{\mathbb{R}},{1}_{\mathbb{R}})=0$ by \ref{fact} and hence  $\theta'_{A,}(c_{m,P}([rs],[sr],[s],[r])=0$ as required. 
 
  The proof of (2) is similar but simpler. Let  $A$ be  a SRRM. Since $R_{W,S} = P_{W,S}/  I_{W,S}$, it will suffice to show that $\ker \theta'_{A}\supseteq  I_{W,S}$. Recall that $ I_{W,S}$ is generated by elements $C_{m}([rs],[sr],[s],[r])\in P_{W,S}$ where   $r,s\in S$ with $3\leq m:=m_{r,s}<\infty$. 
Since $\varphi(m)$ is even, we have\[\begin{split}\theta'_{A,}(C_{m,P}([rs],[sr],[s],[r])&=C_{m,B}(-a_{r,s}e_{r,s},-a_{sr}e_{s,r},e_{s,s},e_{r,r})=c_{m,\mathbb{R}}(-a_{r,s},-a_{sr},1_{\mathbb{R}},{1}_{\mathbb{R}})e_{r,r}\end{split}.\]
Since $A$ is a SRRM, we may write 
 $a_{r,s}a_{s,r}=4\cos^{2}\frac{k\pi}{m}$ where the integer $k$   is relatively prime to $m$ and $1\leq k< m$. We then  have $C_{m,-\mathbb{R}}(-a_{r,s},-a_{sr},1_{\mathbb{R}},{1}_{\mathbb{R}})=h_{m}\left(4\cos^{2}\frac{k\pi}{m}\right)=0$  and hence  $\theta'_{A,}(c_{m,P}([rs],[sr],[s],[r])=0$ as required. 

In (3)--(4), we regard $\mathrm{Mat}'_{J\times J}(\mathbb{R})$, for $J\subseteq S$,  as a $\mathbb{R}$-subalgebra of $B$ in the natural way (identifying its matrix units $e_{j,k}$, for $j,k\in J$, with $e_{j,k}\in B$). To prove (3), observe that  by (1), 
 $\mathrm{Img}(\theta_{A,\widetilde R})=\mathrm{Img}(\theta'_{A})$. This image has the same $\mathbb{R}$-linear span as the set of matrix units $e_{r,s}$
 such that there  is some path in $\Gamma_{W,S}$ with successive vertices $r=s_{1},\ldots, s_{n}=s$ such that $a_{s_{i},s_{i+1}}\neq 0$ for $i=1,\ldots, n-1$. Since $A$ is a LRRM, for any  $t\neq u\in S$, we have  $a_{t,u}\neq 0\iff a_{u,t}\neq 0\implies m_{t,u}\neq 2$, and (3) follows.
 
 Part (4) follows from (3) since if $A$ is a SRRM, then $\mathrm{Img}(\theta_{A,R})=\mathrm{Img}(\theta_{A,\widetilde R})$ and for distinct 
$t,u\in S$, we have  $a_{t,u}\neq 0\iff a_{u,t}\neq 0\iff m_{t,u}\neq 2$.
\end{proof}

\subsection{} By base change $-\otimes_{\mathbb{Z}}\mathbb{R}$, we obtain from the ring $\widetilde R_{W,S}$ and  ring homomorphism $\theta_{A,\widetilde R}\colon \widetilde R_{W,S}\to B$ into the $\mathbb{R}$-algebra $B$, 
a $\mathbb{R}$-algebra  $\widetilde{{R}}':=\widetilde{R}_{W,S}\otimes_{\mathbb{Z}}
\mathbb{R}$ and $\mathbb{R}$-algebra homomorphism
$\theta_{A,\widetilde R'}\colon \widetilde R'\to B$. Similarly,  define the $\mathbb{R}$-algebra  ${{R}}':={R_{W,S}}\otimes_{\mathbb{Z}}
\mathbb{R}$ and $\mathbb{R}$-algebra homomorphism
$\theta_{A, R'}\colon  R'\to B$. We have natural inclusion homomorphisms $\widetilde R_{W,S}\to\widetilde R'$ and $R_{W,S}\to R'$ of rings, given in each case by $r\to r\otimes_{\mathbb{Z}}1$, since $\widetilde R_{W,S}$ and $R_{W,S}$ are free as $\mathbb{Z}$-modules.

Note that if $C$ denotes one of the rings (or $\mathbb{R}$-algebras) $\widetilde R_{W,S}$, $\widetilde R'$,  $R_{W,S}$, $R'$ or $B$, we have defined on $C$  commuting left and right $W$ actions, such that the left (resp., right)  $W$-action is right (resp., left) 
$C$-linear.  We regard each such $C$ as a $W\times W^{\mathrm{op}}$-set in the natural way.  The two inclusion homomorphisms $\widetilde R_{W,S}\hookrightarrow\widetilde R'$ and $R_{W,S}\hookrightarrow R'$, and the canonical surjective homorphisms $\widetilde R_{W,S}\twoheadrightarrow R$
and $\widetilde R'_{W,S}\twoheadrightarrow R'$   are obviously   $W\times W^{\mathrm{op}}$-equivariant.

 \begin{theorem}\label{WxWaction}  Let notations and assumptions be as above. 
\begin{enumerate}\item If $A$ is a LRRM for $(W,S)$, the ring homomorphism $\theta_{A,{\widetilde R}}\colon \widetilde R_{W,S}\to B$  and $\mathbb{R}$-algebra homomorphism
 $\theta_{A,{\widetilde R}'}\colon \widetilde R'\to B$ 
   are $W\times W^{\mathrm{op}}$-equivariant.
   \item If $A$ is a SRRM for $(W,S)$, the ring homomorphism   $\theta_{A,{R}}\colon  R_{W,S}\to B$   and $\mathbb{R}$-algebra homomorphism $\theta_{A,{ R}'}\colon  R'\to B$ are
   $W\times W^{\mathrm{op}}$-equivariant.
    \item Suppose that  $(W,S)$ is irreducible and  $A$ is a SRRM. Then   $\theta_{A,R'}$ is surjective.  
    \item The $W\times W^{\mathrm{op}}$ actions on $R_{W,S}$,  $R'$  $\widetilde R_{W,S}$, and $\widetilde R'$ are all faithful.
\end{enumerate}
\end{theorem}

\begin{proof} For  all  the various left and right $W$-sets $X$ involved in this proof, we   write the various action maps $W\times X\to X$ as $(w,x)\mapsto wx$ and  $X\times W\to X$ as $(x,w)\mapsto xw$ for notational simplicity.  
To prove (1), it will suffice to prove the assertion for $\theta_{A,\widetilde R}$, as that for  $\theta_{A,\widetilde R'}$ then follows by base change. We show first that for $r,s\in S$,  we have
$\theta_{A,{\widetilde {R}}}([r]s)=\left(\theta_{A,{\widetilde {R}}}([r])\right)s$. If  
$s=r$,  then  \[\theta_{A,{\widetilde {R}}}([r]s)=\theta_{A,{\widetilde {R}}}(-[r])=-e_{r,r}=e_{r,r}r=\theta_{A,{\widetilde {R}}}([r])s\] while if $s\neq r$, then 
\[\theta_{A,{\widetilde {R}}}([r]s)=\theta_{A,{\widetilde {R}}}([r]+[rs])=
e_{r,r}-a_{r,s}e_{r,s}=e_{r,r}s=\theta_{A,{\widetilde {R}}}([r])s,\]  as required.
Now for any $x\in R$ and  $r,s\in S$ we have \begin{equation*}\begin{split}
\theta_{A,{\widetilde {R}}}((x\cdot [r])s)&=\theta_{A,{\widetilde {R}}}(x\cdot ([r]s))=\theta_{A,{\widetilde {R}}}(x)\cdot 
\theta_{A,{\widetilde {R}}}([r]s)=\theta_{A,{\widetilde {R}}}(x)\cdot
(\theta_{A,{\widetilde {R}}}([r])s)\\&=\left(\theta_{A,{\widetilde {R}}}(x) 
\cdot\theta_{A,{\widetilde {R}}}([r])\right)s=
\left(\theta_{A,{\widetilde {R}}}(x\cdot [r])\right)s\end{split}
\end{equation*}
Since $x=\sum_{r\in R}x\cdot [r]$ where $x\cdot [r]=0$ for almost all $r\in S$, this impies 
\begin{equation*}\begin{split}
\theta_{A,{\widetilde {R}}}(xs)&=\theta_{A,{\widetilde {R}}}\left(\sum_{r\in S}\left(x\cdot [r]\right)s\right)=
\sum_{r\in S}\theta_{A,{\widetilde {R}}}\left((x\cdot[r])s\right)\\ &=\sum_{r\in S}\theta_{A,{\widetilde {R}}}(x\cdot[r])s=\theta_{A,{\widetilde {R}}}\left(\sum_{r\in S}(x\cdot [r])\right )s=\theta_{A,{\widetilde {R}}}(x)s
\end{split}
\end{equation*}
Finally, by induction on the length of $w\in W$, it follows 
$\theta_{A,{\widetilde {R}}}(xw)=\theta_{A,{\widetilde {R}}}(x)w$ for all $x\in \widetilde R_{W,S}$ and $w\in W$. Similarly or by symmetry,  
$\theta_{A,{\widetilde {R}}}(wx)=w\theta_{A,{\widetilde {R}}}(x)$ for all $x\in \widetilde R_{W,S}$ and $w\in W$completing the proof of (1).

Part (2) follows from (1) since if  $A$ is a SRRM, we have the factorization of  $\theta_{A,\widetilde R}$ in Proposition \ref{homsm} where the canonical surjection 
$\widetilde{R}_{W,S}\twoheadrightarrow R_{W,S}$ is $W\times W^{\mathrm{op}}$-equivariant. Part (3) follows from Proposition \ref{homsm} (4).

We prove (4). There is no loss of generality in taking $(W,S)$ to be irreducible.  Let $A$  be a NGCM for $(W,S)$. It is known
that the  $W$ actions on $V_{A}$ and $\ck V_{A}$ are faithful, so the induced $W\times W^{\mathrm{op}}$-action  on $\ck V_{A}\otimes_{\mathbb{R}} V_{A}\cong B$ is faithful. 
But $B$ is a quotient of  $R'$, so  the  $W\times W^{\mathrm{op}}$-action  on $R'$ is faithful. Since that action arises by base change from that on $R_{W,S}$, the action on $R_{W,S}$ is faithful. Since $R_{W,S}$ (resp., $R'$) is a quotient of $\widetilde R_{W,S}$ (resp., $\widetilde R'$), all the required    $W\times W^{\mathrm{op}}$-actions are faithful.
 \end{proof}
 The above theorem has many consequences for special classes of  Coxeter groups, such as  finite Coxeter groups groups and affine Weyl groups, amongst which we observe here only the following.    
\begin{corollary}  Suppose that $(W,S)$ is an irreducible finite Weyl group with Cartan matrix $A$. Then $R'$ is $W\times W^{\mathrm{op}}$-equivariantly isomorphic to  the matrix ring $B=\mathrm{Mat}_{S\times S}(\mathbb{R})$ as $\mathbb{R}$-algebra, where   $B$ may be   $W\times W^{\mathrm{op}}$-equivariantly identified with  $\ck V_{A}\otimes_{\mathbb{R}}V_{A}$ so that $\ck\alpha_{r}\otimes \alpha_{s}= e_{r,s}$. \end{corollary} 
\begin{proof}
We prove  this under the  more general hypotheses that 
$(W,S)$ is of finite rank, that the Coxeter graph of $(W,S)$ is a tree and that $m_{r,s}\in \{2,3,4,6\}$ 
for all $r\neq s$ in $S$.\footnote{The finite rank assumption can also be omitted but we leave the additional arguments to show this to the reader.}
  By (2) and Example \ref{Weyl}(1),
$\theta_{A,R'}\colon R'\to B=\mathrm{Mat}_{S\times S}(\mathbb{R})$ is a surjective, $W\times W^{\mathrm{op}}$-equivariant homomorphism of $\mathbb{R}$-algebras of the same finite dimension $\vert S\vert ^{2}$. Hence it is an isomorphism as asserted. 
\end{proof}

 \section{Free Products of Rings}\label{fpr}
 In this section, we give a summary of definitions and results on the free product of rings from Cohn \cite{Cohn1}, \cite{Cohnfir}, and \cite{Cohn2}. We apply the results to show that a ring associated to an edge of the graph of a Coxeter system is a free product of rings.

\subsection{}  Let $B$ be an associative ring with a unit element and let $A$ be a (not necessarily commutative) subring of $B$ containing the unit element,
then we say $B$ is an $A$ ring. (Note that if $A$ is commutative and $B$ is a unital  $A$-algebra 
in which $A$ embeds by $a \mapsto a1_B$,  then $B$ is an $A$ ring). 
A homomorphism of $A$ rings $f:B\to C$ is an $A$-bimodule homomorphism that sends the unit element of $B$ to the unit element of $C$.

\begin{definition} (Cohn \cite{Cohn1}, Section 3) \label{deffpr}The $A$ ring $R$  is said to be the free product of the $A$ rings  $\{R_i | i \in I\}$ ( over $A$) if 
 $\{R_i | i \in I\}$   forms a family of subrings of $R$ such that 
\begin{enumerate}[label=(\roman*)]
\item $R_i \cap R_j = A$ for $i \not= j$, 
\item if $X_i$ is a set of generators of $R_i$, $i \in I$, then $\bigcup_IX_i$ is a set of generators of $R$,
\item If $C_i, i \in I$ is a set of defining relations of $R_i, i \in I$ (in terms of the set of generators $X_i$) then $\bigcup_IC_i$ is a set of defining relations of $R$ (in terms of the set of generators $\bigcup_IX_i$). 
\end{enumerate}
\end{definition}

The free product of $A$ rings does  not always exist, but when it does it coincides with 
the universal product which is always guaranteed to exist, see Cohn \cite{Cohn1}, Theorem 3.1.
When the free product of the $A$-rings $\{R_i | i \in I\}$ exists, we will denote it by $\displaystyle \{\ast_AR_i\}_{i \in I}$ or $R_1\ast_A R_2$ if $I = \{1, 2\}$.

\begin{definition}An $A$ ring $U$ is called the universal product of the $A$ rings  $\{R_i |\  i \in I\}$ if
there exist homomorphisms $\phi_i:R_i \to U$ such that 
\begin{enumerate}[label=(\roman*)]
\item $U$ is generated by the subrings $\phi_i(R_i)$, 
\item given any $A$ ring, $C$, and any family of homomorphisms $\psi_i: R_i \to C$, there exists 
a homomorphism $\psi : U \to C$ satisfying $\psi \circ \phi_i = \psi_i$ for each $i \in I$.
\end{enumerate}
\end{definition}

One can see that the universal product of $A$ rings is determined up to isomorphism by the universal mapping property. It is not difficult to see that the free product, when it exists, has the same universal mapping property and thus must be isomorphic to the universal product. The following theorem gives the precise relationship between the two:

\begin{theorem} (Cohn \cite{Cohn1})\label{cohn} Let $R_i, i \in I$ be a family of $A$ rings and $\{U, \phi_i:R_i \to U| \ i \in I\}$ their universal product. Then the free product of the $R_i, i \in I$ exists if and only if 
\begin{enumerate}[label=(\roman*)]
\item the canonical homomorphism $\phi_i:R_i \to U$ is a monomorphism for each $i \in I$, 
\item $\phi_i(R_i)\cap \phi_j(R_j) = A$ for $i \not= j$.
\end{enumerate}
When these two conditions hold, $U$ is in fact the free product of the rings $R_i$.
\end{theorem}

In Cohn \cite{Cohn2} (also see Cohn \cite{Cohnfir} and Serre \cite{Serre}, Exercise 1.2. for special cases), the construction of the free product of 
a family of $A$ rings,  using 
tensor products of bimodules,  is described. The free product   exists when 
some conditions of flatness are satisfied, see Cohn \cite{Cohn1}, Theorem 4.5. 
In particular, when  the base ring, $A$,  is a field the free product of any family of $A$ rings exists, 
Cohn \cite{Cohn1}, Corollary to Theorem 4.7.

Using the construction of the free product via tensor products of bimodules, Cohn proves the following:
\begin{theorem}(Cohn \cite{Cohn2}, Theorem 2.5) \label{zerodiv} If $\{R_i | i \in I\}$ is any family of $K$-rings without zero divisors, where $K$ is a field, then their free product has no zero divisors. 
\end{theorem}

We use Cohn's results to gain insight into  the structure of the following rings, which will in turn help us to understand the structure of the path algebras we have associated to Coxeter systems.

\begin{theorem}\label{polyrings} Let $f(t)$ be a non-constant polynomial in $\mathbb{Q}[t]$ with non-zero constant term, and let $\mathbb{Q}[y, \bar{y}]$ be the ring of polynomials 
over $\mathbb{Q}$ in the non-commuting variables $y$ and $\bar{y}$. Let $R_f$ denote the quotient ring;
$$R_f = \mathbb{Q}[y, \bar{y}]/\langle f(y\bar{y}), f(\bar{y}y)\rangle.$$
 Then $R_f$ is isomorphic to the ring 
$$\mathbb{Q}[x_y, x_y^{-1}] \ast_{\mathbb{Q}} K,$$
where $\mathbb{Q}[x_y, x_y^{-1}]$ is the ring of Laurent polynomials in the variable $x_y$ over $\mathbb{Q}$ and $K$ is  the ring extension $\mathbb{Q}[t]/\langle f(t)\rangle$ of $\mathbb{Q}$. 
Furthermore, if $f(t)$ is irreducible, then $R_f$ is a domain.
\end{theorem}

\begin{proof} If $f(t)$ is irreducible, then $K$ is a field extension of $\mathbb{Q}$ and hence, is a domain. 
That the ring of Laurent polynomials, $\mathbb{Q}[x_y, x_y^{-1}]$, is a domain is well known.
By Theorem \ref{zerodiv}, the ring $\mathbb{Q}[x_y, x_y^{-1}] \ast_{\mathbb{Q}} K$ does not have zero divisors, since 
both $\mathbb{Q}[x_y, x_y^{-1}]$ and $K = \mathbb{Q}[t]/\langle f(t)\rangle$ are domains. Hence it is enough to prove that $R_f$ is isomorphic to the free product
$$\mathbb{Q}[x_y, x_y^{-1}] \ast_{\mathbb{Q}} K.$$

First we claim that $y$ has an inverse in the ring $R_f = \mathbb{Q}[y, \bar{y}]/\langle f(y\bar{y}), f(\bar{y}y)\rangle$. Since $f$  is non-constant  with a non-zero constant term, $f(t) =  th(t) + \kappa = h(t)t + \kappa$ for some $\kappa \in \mathbb{Q}, \kappa \not=0$ and some $h(t) \in \mathbb{Q}[t]$. 
Without loss of generality, we can assume that $\kappa = -1$. Thus, 
$$y\bar{y}h(y\bar{y}) = 1_{R_f} = h(y\bar{y})y\bar{y},$$
and $\displaystyle (y\bar{y})^{-1} = h(y\bar{y})$ in the ring $R_f$. 
By symmetry, $\bar{y}y$ has an inverse in 
$R_f$ and  $\displaystyle (\bar{y}y)^{-1} = h(\bar{y}y)$.
 Using our expression for $ (y\bar{y})^{-1}$, we see that $\displaystyle \bar{y}h(y\bar{y})$ is a right inverse for $y$ and using our expression for $(\bar{y}y)^{-1}$, we see that 
$$1_{R_f} =  h(\bar{y}y)\bar{y}y =  \bar{y}h(y\bar{y})y,$$
which shows that $\displaystyle  \bar{y}h(y\bar{y})$ is also a left inverse for $y$ in $R_f$. This proves our claim.

Now let $\hat{\phi}:\mathbb{Q}[y, \bar{y}] \to \mathbb{Q}[x_y, x_y^{-1}] \ast_{\mathbb{Q}} K$ be the unique $\mathbb{Q}$ algebra homomorphism with
$\hat{\phi}(y) = x_y$ and $\hat{\phi}(\bar{y}) = x_y^{-1}t$. This induces  a well defined quotient homomorphism, 
 $\phi: R_f \to \mathbb{Q}[x_y, x_y^{-1}] \ast_{\mathbb{Q}} K$  since $\hat{\phi}(f(y\bar{y})) = f(\hat{\phi}(y\bar{y})) = f(x_yx_y^{-1}t) = f(t) = 0$ and $\hat{\phi}(f(\bar{y}y)) = f(\hat{\phi}(\bar{y}y))= f(x_y^{-1}tx_y)
= x_y^{-1}f(t)x_y = 0$.

To show that the map $\phi$ defined above is an isomorphism, we construct its inverse.
We let $\hat{\psi}_1: \mathbb{Q}[t] \to R_f$ be the $\mathbb{Q}$ algebra homomorphism such that $\hat{\psi}_1(t) = y\bar{y}$.
Since $\hat{\psi}_1(f(t)) = f(\hat{\psi}_1(t)) = f(y\bar{y}) = 0$, we get a well defined quotient homomorphism
 $\psi_1: K \to R_f$.  Let  $\psi_2:\mathbb{Q}[x_y, x_y^{-1}] \to R_f$ be the $\mathbb{Q}$ algebra homomorphism with  $\psi_2(x_y) = y$ and $\displaystyle \psi_2(x_y^{-1}) = 
\bar{y}h(y\bar{y}).$ This is a well defined  homomorphism since $\psi_2(x_y^{-1}) = (\psi_2(x_y))^{-1}$. 
The universal mapping property of the free product gives a ring homomorphism $\psi: \mathbb{Q}[x_y, x_y^{-1}] \ast_{\mathbb{Q}} K \to R_f$ which restricts to $\psi_1$ and $\psi_2$ on $K$ and $\mathbb{Q}[x_y, x_y^{-1}]$ respectively.

We finish the proof by showing that the homomorphism $\psi$ is the inverse of the homomorphism $\phi$. 
We need only verify that $\psi \circ \phi$ is the identity on the generators of $R_f$ and that $\phi \circ \psi$ is the identity on the generators of $\mathbb{Q}[x_y, x_y^{-1}] \ast_{\mathbb{Q}} K$. We have $\psi(\phi(y)) = \psi(x_y) = y$ and 
$$\psi(\phi(\bar{y})) = \psi(x_y^{-1}t) = \bar{y}h(y\bar{y})y\bar{y} = y^{-1}y\bar{y} = \bar{y}.$$
 Thus $\psi \circ \phi$ is the identity map on $R_f$. 
On the other hand, $\phi(\psi(t)) = \phi(y\bar{y}) = x_yx_y^{-1}t = t$ and  $\phi(\psi(x_y)) = \phi(y) = x_y$. Finally we have 
$$ \phi(\psi(x_y^{-1})) = \phi\left(\bar{y}h(y\bar{y})\right) = x_y^{-1}th(x_yx_y^{-1}t) = x_y^{-1}th(t) = x_y^{-1},$$
since $f(t) = th(t) - 1$ and hence $\displaystyle th(t) = 1_{\mathbb{Q}[x_y, x_y^{-1}] \ast_{\mathbb{Q}} K}$. 
This shows that the rings are isomorphic and completes our proof. 
\end{proof}

\subsection{} Consider now a  graph $\Gamma_{W, S}$ associated to a Coxeter system $(W, S)$  with generators $S = \{s_1, s_2, \dots , s_N\}$ and Coxeter matrix $(m_{ij})_{\{1 \leq i, j \leq N\}}$.
Let $Y_{W, S} =  Y_{W, S, +} \bigcupdot \overline{Y_{W, S, +}}$ be an orientation of  $\Gamma_{W, S}$. 
To each  $y_{ij} = [s_is_j] \in Y_{W, S, +} $,  we associate a quotient ring,  $R_{y_{ij}}$, 
of the polynomial ring $\mathbb{Q}[x_{y_{ij}}, \bar{x}_{{y}_{ij}}]$ in the non-commuting variables $x_{y_{ij}}$ and $\bar{x}_{{y}_{ij}}$ , as follows:
\[
R_{y_{ij}} = \begin{cases} \mathbb{Q}[x_{y_{ij}}, \bar{x}_{{y}_{ij}}] & \mbox{if} \ m_{ij} = \infty\\
\mathbb{Q}[x_{y_{ij}}, \bar{x}_{{y}_{ij}}]/\langle C_{m_{ij}}(x_{y_{ij}}\bar{x}_{y_{ij}}), C_{m_{ij}}(\bar{x}_{y_{ij}}x_{y_{ij}})  \rangle & \mbox{if} \ m_{ij} < \infty\end{cases}
\]

\begin{corollary}\label{cor} Let $\Gamma_{W, S}$ be a  graph  associated to a Coxeter system $(W, S)$  with generators $S = \{s_1, s_2, \dots , s_N\}$ and Coxeter matrix $(m_{ij})_{\{1 \leq i, j \leq N\}}$.
 Let $R_{y_{ij}}$ be the ring associated to the edge 
$y_{ij} = [s_is_j] \in Y_{W, S, +}$. Then we have the following:
\[R_{y_{ij}} \cong \begin{cases}
 \mathbb{Q}[x_{y_{ij}}, \bar{x}_{y_{ij}}] &\mbox{if} \  m_{ij} = \infty\\
 \mathbb{Q}[x_{y_{ij}}, x_{y_{ij}}^{-1}]\ast_{\mathbb{Q}}K &\mbox{if} \ m_{ij} < \infty
 \end{cases},
 \]
 where $ \mathbb{Q}[x_{y_{ij}}, x_{y_{ij}}^{-1}]$ is the ring of Laurent polynomials in $x_{y_{ij}}$ over $\mathbb{Q}$ and $K \cong \mathbb{Q}[t]/\langle C_{m_{ij}}(t)\rangle$ is a field extension of $\mathbb{Q}$. 
Furthermore, $R_{y_{ij}}$ is a domain for all $y_{ij}\in Y_{W, S, +}$. 
\end{corollary}

\begin{proof} This follows directly from Theorem \ref{polyrings} and the fact that the ring $\mathbb{Q}[x_{y_{ij}}, \bar{x}_{y_{ij}}]$ in the non-commuting variables $x_{y_{ij}}$ and  $\bar{x}_{y_{ij}}$ is a domain (see, for example, 
Cohn \cite{Cohnund}, Section 5.3).
\end{proof}

 \section{A faithful Representation of  $\hat{R}_{W, S}$.}\label{imbed}
 In this section, we show that the  algebra  $\hat{R}_{W, S}$(defined in Section \ref{sec5}) associated to a finite  rank Coxeter system  is isomorphic to  a subring of a matrix ring over a non-commutative ring.

 \subsection{} Throughout the section   $\Gamma_{W, S}$ denotes  a graph    associated  to  a Coxeter system  $(W, S)$ with generators $S = \{s_1, \dots , s_N\}$ and Coxeter matrix $(m_{ij})_{\{1 \leq i, j \leq N\}}$.
We fix an orientation, $Y_{W, S} =  Y_{W, S, +} \bigcupdot \overline{Y_{W, S, +}}$,  for our graph $\Gamma_{W, S}$. 
We will show that the associated algebra $\hat{R}_{W, S}$ can be imbedded as a $\mathbb{Q}$ algebra into 
a matrix ring $M_{N+1}(Q)$, where $Q$ is a quotient of a non-commutative polynomial ring over $\mathbb{Q}$ in several variables. This is turn will give us some insight into the zero divisors in $\hat{R}_{W, S}$ and its subring $R_{W, S}$. 

\subsection{A Graph Extension} Let $(W^{\bullet}, S^{\bullet})$ denote the Coxeter system with Coxeter generators 
$S^{\bullet}$
 $= \{s_1,  \dots  ,  s_N,   s_{N + 1}\}$ and Coxeter matrix $(m^{\bullet}_{ij})_{1 \leq i, j \leq N+1}$, such that $m_{ij}^{\bullet} = m_{ij}, 1\leq i, j\leq N$,  $m^{\bullet}_{\underline{N+1} i} = m^{\bullet}_{i \underline{N+1}} = 3, 1\leq i \leq N$ and $m^{\bullet}_{\underline{N+1}\  \underline{N+1}} = 1$. (We denote $N+1$ by $\underline{N+1}$ in subscripts where  confusion  might be caused by ambiguity.)
 Since $W$ can be identified with  the  parabolic subgroup  $W^{\bullet}_{[N]}$ of $W^{\bullet}$ corresponding to the Coxeter system $(W^{\bullet}_{[N]}, S^{\bullet}_{[N]})$, by Lemma \ref{impar}, we have an inclusion of the graph $\Gamma_{W, S}$ in the graph $\Gamma_{W^{\bullet}, S^{\bullet}}$ and  imbeddings of the associated non-unital $\mathbb{Q}$ algebras  ${i}_{[N], [N + 1]}: \hat{P}_{W, S} \to  \hat{P}_{W^{\bullet}, S^{\bullet}}$
 and 
 $\tilde{i}_{[N], [N + 1]}: \hat{R}_{W, S} \to  \hat{R}_{W^{\bullet}, S^{\bullet}}$. For simplicity of notation, we will identify the vertices and edges of 
 the graph $\Gamma_{W, S}$ with their images in the graph of $\Gamma_{W^{\bullet}, S^{\bullet}}$. We can extend the orientation of $\Gamma_{W, S}$
 to an orientation  for $\Gamma_{W^{\bullet}, S^{\bullet}}$ by letting $ Y_{W^{\bullet}, S^{\bullet}, +} =  Y_{W, S, +} \bigcupdot \{y_{\underline{N+1} i} | 1\leq i \leq N\}$.

  \begin{exmp} For our running example, 
 we let $S^{\bullet}= \{r, s, t, u, v, w \}$ with new vertex $w$. 
 We show the positively oriented edges, $Y_{W^{\bullet}, S^{\bullet}, +}$, of the graph 
 $\Gamma_{W^{\bullet}, S^{\bullet}, +}$ below.
 The vertices and edges in   $\Gamma_{W^{\bullet}, S^{\bullet}}\backslash \Gamma_{W, S}$ are shown in blue. Alongside, we show the new vertex and new geometric edges  of $\Gamma_{W^{\bullet}, S^{\bullet}}$  with their  labels in blue.
 
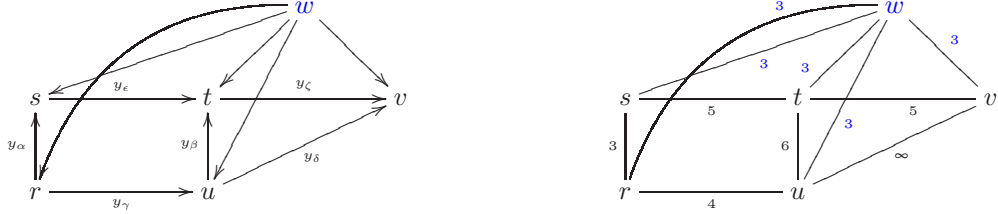
\[
\xymatrix{
&&&{\color{blue} w}\ar@{->}@[blue][dlll]\ar@{->}@[blue][dl]\ar@{->}@[blue][dr]\ar@{->}@[blue][ddl]\ar@/_2pc/@[blue][ddlll]\\
s\ar@{->}[rr]^{y_{\epsilon}}&&t\ar@{->}[rr]^{y_\zeta}&&{v}&&\\
{r}\ar@{->}[u]^{y_\alpha}\ar@{->}[rr]_{y_\gamma}&&{u}\ar@{->}[u]^{y_\beta}\ar@{->}[urr]_{y_\delta}&&\\
}
\quad
\xymatrix{
&&&{\color{blue} w}\ar@{-}@[blue][dlll]^{\color{blue}3}\ar@{-}@[blue][dl]_(.8){\color{blue}3}\ar@{-}@[blue][dr]^{\color{blue}3}\ar@{-}@[blue][ddl]^(.6){\color{blue}3}\ar@{-}@/_2pc/@[blue][ddlll]_(.3){\color{blue}3}\\
s\ar@{-}[rr]_{5}&&t\ar@{-}[rr]_(.6){5}&&{v}&&\\
{r}\ar@{-}[u]^{3}\ar@{-}[rr]_{4}&&{u}\ar@{-}[u]^{6}\ar@{-}[urr]_{\infty}&&\\
}
\]
 \end{exmp}

 \subsection{Isomorphism of $\hat{R}_{W^{\bullet}, S^{\bullet}}$ and a Matrix Ring} 
 For $y \in Y_{W, S, +}$, let $R_y$ be the associated ring described at the end of  Section \ref{fpr}. 
 By Definition \ref{deffpr}, the free product  $\left\{\ast_{\mathbb{Q}}R_y\right\}_{y \in Y_{W, S, +}}$
 is isomorphic to the quotient ring, $Q$,  of the polynomial ring in the non-commuting variables $\left\{x_y, \bar{x}_y  | \ y  \in Y_{W, S, +}\right\}$;
 $$Q = \mathbb{Q}[\left\{x_{y_{ij}}, \bar{x}_{y_{ij}}\right\}_{y_{ij}  \in Y_{W, S, +}}]/\left\langle \left\{C_{m_{ij}}(x_{y_{ij}}\bar{x}_{y_{ij}}), C_{m_{ij}}(\bar{x}_{y_{ij}}x_{y_{ij}})\right\}_{\left\{ y_{ij}  \in Y_{W, S, +} | \ m_{ij} < \infty\right\}} \right\rangle.$$ 
 By Theorem \ref{zerodiv} and Corollary \ref{cor}, $Q$ is a domain.

 We will denote the coset, $p + \hat{I}_{W^{\bullet}, S^{\bullet}}$, of $p \in \hat{P}_{W^{\bullet}, S^{\bullet}}$ in $\hat{R}_{W^{\bullet}, S^{\bullet}}$ 
  by $p^{\sim}$ in what follows. We first single out some elements of $\hat{R}_{W^{\bullet}, S^{\bullet}}$ which will play a central role  in defining a homomorphism from $\hat{R}_{W^{\bullet}, S^{\bullet}}$ to $M_{N + 1}(Q)$. 
 For $1 \leq i \leq N + 1$, we let 
 \begin{equation}
 P_i^j = \begin{cases}
 [s_is_{N+1}s_j]^{\sim} = (y_{i \underline{N+1}}y_{\underline{N + 1}  j})^{\sim} &\mbox{if} \ 1 \leq i, j \leq N\\
 [s_is_{N+1}]^{\sim} = (y_{i \underline{N+1}})^{\sim} &\mbox{if} \ 1 \leq i \leq N,  j = N + 1\\
 [s_{N + 1}s_j]^{\sim} =  y_{\underline{N+1} j}^{\sim} &\mbox{if} \   i = N + 1, 1 \leq j \leq N\\
 [s_{N + 1}]^{\sim} &\mbox{if}\  i = j = N + 1\\
 \end{cases}
 \end{equation}
 
\begin{note} In what follows we use the notation $\dot{o}(y)$ and $\dot{t}(y)$ to denote the index of $o(y)$ and $t(y)$ respectively in the set $[N + 1]$, 
for $y \in Y_{W^{\bullet}, S^{\bullet}}$, i.e. if $y = [s_i, s_j]$, then $\dot{o}(y) = i$ and $\dot{t}(y) = j$. 
\end{note}
 
  \begin{lemma}\label{A} We have the following 
identities  in $\hat{R}_{W^{\bullet}, S^{\bullet}}$, where $s_i, s_j, s_k, s_l \in S^{\bullet}$:
$$\begin{aligned}
    P_{i}^{i} &= [s_i]^{\sim}\\
     P_{i}^{j}P_{k}^{l} &= 0^{\sim} \  \mbox{if} \  j \not= k \\
      P_{i}^{j}P_{k}^{l}  &=   P_{i}^{l}   \  \mbox{if} \ j = k \\
     \end{aligned}$$
 \end{lemma}
 
 \begin{proof}Note that  $C_{y_{ij}}^{\sim} = [s_is_js_i]^{\sim}-[s_i]^{\sim} = 0^{\sim}$ if either $i$ or $j$ are equal to $N + 1$. 
  Hence, if $i \not= N + 1$, then  $P_i^i = [s_is_{N + 1}s_i]^{\sim} = [s_i]^{\sim}$.
 If $i = N+1$, then $P_i^i = [s_{N + 1}]^{\sim}$. This proves the first identity.

 Consider now   $P_{i}^{j}P_{k}^{l}$. It is clear that the product is $0^{\sim}$ if $j \not= k$. 
 Suppose that $1 \leq i, j, k \leq N$, then 
 $$P_i^jP_j^k = ([s_is_{N+1}s_j][s_js_{N + 1}s_k])^{\sim} = ([s_is_{N + 1}][s_{N + 1}s_js_{N + 1}][s_{N + 1}s_k])^{\sim} = [s_is_{N + 1}s_k]^{\sim} = P_i^k$$
 by our opening remark. The other cases follow similarly and are left to the reader.
 \end{proof}

\subsection{} For each $i \in \{1, 2, \dots , N + 1\}$, we use the universal mapping property of the polynomial ring to define a ring homomorphism  $\hat{\psi}_i: \mathbb{Q}[\{x_y, \bar{x}_y\}_{y \in Y_{W, S, +}}] \to [s_i]^{\sim}\hat{R}_{W^{\bullet}, S^{\bullet}}[s_i]^{\sim} \subset \hat{R}_{W^{\bullet}, S^{\bullet}}$, such that 
$\hat{\psi}_i(1_Q) = [s_i]^{\sim}$, 
$\hat{\psi}_i(x_y) = P_i^{\dot{o}(y)}y^{\sim}P_{\dot{t}(y)}^i$  and $\hat{\psi}_i(\bar{x}_y) = P_i^{\dot{o}(\bar{y})}\bar{y}^{\sim}P_{\dot{t}(\bar{y})}^i$ for $y \in Y_{W, S,+}$.  Since $o(\bar{y}) = t(y)$ and $t(\bar{y}) = o(y)$, we have 
$$P_{i}^{\dot{o}(y)}y^{\sim}P_{\dot{t}(y)}^iP_i^{\dot{o}(\bar{y})}\bar{y}^{\sim}P_{\dot{t}(\bar{y})}^{i} = P_{i}^{\dot{o}(y)}(y\bar{y})^{\sim}P_{\dot{t}(\bar{y})}^{i}.$$
  Hence   $\hat{\psi}_i((x_y\bar{x}_y)^k) = P_{i}^{\dot{o}(y)}\left((y\bar{y})^k\right)^{\sim}P_{\dot{t}(\bar{y})}^{i}$, by Lemma \ref{A} .
Thus for $y_{ij} \in Y_{W, S, +}$, $m_{ij} < \infty$, we have  
$$ \hat{\psi}_i\left(C_{m_{ij}}(x_{y_{ij}}\bar{x}_{y_{ij}})\right) = C_{m_{ij}}\left(P_i^{\dot{o}(y)}(y\bar{y}\right)^{\sim}P_{\dot{t}(\bar{y})}^i) = P_i^{\dot{o}(y)}C_{m_{ij}}\left((y\bar{y})^{\sim}\right)P_{\dot{t}(\bar{y})}^i = 0^{\sim}$$
 Therefore  the homomorphism $\hat{\psi}_i$ gives us a well defined quotient homomorphism $$ \psi_i: Q \to [s_i]^{\sim}\hat{R}_{W^{\bullet}, S^{\bullet}}[s_i]^{\sim} \subset \hat{R}_{W^{\bullet}, S^{\bullet}}$$  acting as follows on the generators of $Q$; 
 $$
 \begin{aligned}
 \psi_i(1_Q) &= [s_i]^{\sim}&  \\
 {\psi}_i(x_y) &= P_i^{\dot{o}(y)}y^{\sim}P_{\dot{t}(y)}^i,  & y \in Y_{W, S,+}.\\
 {\psi}_i(\bar{x}_y) &= P_i^{\dot{o}(\bar{y})}\bar{y}^{\sim}P_{\dot{t}(\bar{y})}^i, & y \in Y_{W, S,+}.\\
 \end{aligned}
 $$
 \begin{lemma}\label{Pnphi} Let   $\psi_i:Q \to [s_i]^{\sim}\hat{R}_{W^{\bullet}, S^{\bullet}}[s_i]^{\sim} \subset  \hat{R}_{W^{\bullet}, S^{\bullet}}$ and $\psi_j :Q \to [s_j]^{\sim}\hat{R}_{W^{\bullet}, S^{\bullet}}[s_j]^{\sim} \subset  \hat{R}_{W^{\bullet}, S^{\bullet}}$,  \ $1 \leq i, j \leq N + 1$ be as  defined above. We have $P_j^i\psi_i(q)P_i^j = \psi_j(q)$ for all $q \in Q, 1 \leq i, j \leq N + 1$. 
 \end{lemma}
 \begin{proof}It suffices to show equality on the generators $1_Q, x_y, \bar{x}_y, y \in Y_{W, S, +}$ of $Q$. 
 Using the identities in  Lemma \ref{A}, we get:
 $$P_j^i\psi_i(1_Q)P_i^j = P_j^i[s_i]^{\sim}P_i^j = P_j^iP_i^j = P_j^j = [s_j]^{\sim} = \psi_j(1_Q),$$
 $$P_j^i\psi_i(x_y)P_i^j  = P_j^iP_i^{\dot{o}(y)}y^{\sim}P_{\dot{t}(y)}^iP_i^j = P_j^{\dot{o}(y)}y^{\sim}P_{\dot{t}(y)}^j = \psi_j(x_y),$$
 and 
  $$P_j^i\psi_i(\bar{x}_y)P_i^j  = P_j^iP_i^{\dot{o}(\bar{y})}\bar{y}^{\sim}P_{\dot{t}(\bar{y})}^iP_i^j = P_j^{\dot{o}(\bar{y})}\bar{y}^{\sim}P_{\dot{t}(\bar{y})}^j = \psi_j(\bar{x}_y).$$
  This proves the lemma.
  \end{proof}

\subsection{}  We let $\Psi : M_{N + 1}(Q) \to \hat{R}_{W^{\bullet}, S^{\bullet}}$ be the homomorphism of $\mathbb{Q}$ vector spaces which acts on  
 basis elements of $M_{N + 1}(Q)$  of the form $qe_{ij}$ 
 as follows:
 $$\Psi(qe_{ij}) =\psi_i(q)P_i^j,$$
  where $q \in Q$ and $e_{ij}$ is the $(N + 1) \times (N + 1)$ matrix with $1$ in the $(i, j)$ position and zeros elsewhere.
  
 \begin{lemma}  $\Psi$ is a ring homomorphism. 
 \end{lemma}
 \begin{proof}To verify this, we need only show that $\Psi((q_1e_{ij})(q_2e_{kl})) = \Psi(q_1e_{ij})\Psi(q_2e_{kl})$ for $q_1, q_2 \in Q$ and $1 \leq i, j, k, l \leq N + 1$. 
 If $j \not= k$, it is easy to see that both sides are zero and are thus equal. If $j = k$, then,  using Lemma \ref{Pnphi}, we get 
 $$
 \begin{aligned}
 \Psi((q_1e_{ij})(q_2e_{jl})) &= \Psi(q_1q_2e_{il}) = \psi_i(q_1)\psi_i(q_2)P_i^l \\
 &= \psi_i(q_1)P_i^jP_j^i\psi_i(q_2)P_i^jP_j^iP_i^l\\
 &= \psi_i(q_1)P_i^j\psi_j(q_2)P_j^l\\
  &= \Psi(q_1e_{ij})\Psi(q_2e_{jl}).
  \end{aligned}
 $$
 \end{proof}

 We show that both rings are isomorphic by showing that $\Psi$ has an inverse. We define a homomorphism,
 $$\phi: \hat{P}_{W^{\bullet}, S^{\bullet}} \to M_{N + 1}(Q),$$ 
 which acts as follows on the vertices and edges of $\Gamma_{W^{\bullet}, S^{\bullet}}$ as follows:
 $$
 \begin{aligned}
\phi([s_i]) &= e_{ii}, &1 \leq i \leq N + 1,\\
\phi(y_{ij}) &= x_{y_{ij}}e_{ij}, &1 \leq i \not= j \leq N, y_{ij} \in Y_{W^{\bullet}, S^{\bullet}, +}\\
\phi(y_{ij}) &= \bar{x}_{\bar{y}_{ij}}e_{ij}, &1 \leq i \not= j \leq N, y_{ij} \in \overline{Y_{W^{\bullet}, S^{\bullet}, +}}\\
\phi(y_{ij}) &= e_{ij},     &N+1 \in \{i, j\},  i \not=j.\\
    \end{aligned}
    $$
    The homomorphism extends to a  well defined homomorphism on $\hat{P}_{W^{\bullet}, S^{\bullet}}$ since 
    $\phi$ respects the multiplicative relations on the basis of paths $\mathfrak {P}_{W^{\bullet}, S^{\bullet}}$ in $\hat{P}_{W^{\bullet}, S^{\bullet}}$ namely:
    $$
    \begin{aligned}
    \phi(y_{ij})\phi(y_{kl}) &= \phi(0) = 0 &\mbox{if} \ j\not= k, y_{ij}, y_{kl} \in Y_{W^{\bullet}, S^{\bullet}}.\\
    \phi(y_{ij})\phi([s_k]) = \phi([s_k])\phi(y_{ji}) &= \phi(0) = 0 &\mbox{if} \ j\not= k, y_{ij} \in Y_{W^{\bullet}, S^{\bullet}}. \\
    \phi(y_{ij})\phi([s_j]) = \phi([s_i])\phi(y_{ij}) &= \phi(y_{ij})   &\mbox{if} \   y_{ij} \in Y_{W^{\bullet}, S^{\bullet}}.\\
    \end{aligned}
    $$

    The following lemma shows that the quotient homomorphism 
    $\Phi : \hat{R}_{W^{\bullet}, S^{\bullet}}  \to M_{N + 1}(Q)$ given by 
    $$\Phi( p^{\sim}) = \phi(p), \ p \in \hat{P}_{W^{\bullet}, S^{\bullet}}$$
    is well defined:
    
    \begin{lemma} Let $\phi: \hat{P}_{W^{\bullet}, S^{\bullet}} \to M_{N + 1}(Q)$ be as defined above. Then, 
    $\hat{I}_{W^{\bullet}, S^{\bullet}} \subseteq \ker \phi$.
    \end{lemma}
    \begin{proof}It suffices to prove that the generators of $\hat{I}_{W^{\bullet}, S^{\bullet}}$ are in $\ker \phi$.
    For $1 \leq i \not= j \leq N$, $y_{ij} \in Y_{W^{\bullet}, S^{\bullet}, +}$ we have 
     $$\phi(y_{ij}y_{ji}) = \phi(y_{ij})\phi(y_{ji}) = x_{y_{ij}}e_{ij}\bar{x}_{y_{ij}}e_{ji} = x_{y_{ij}}\bar{x}_{{y}_{ij}}e_{ii},$$
    and if $m_{ij} < \infty$, 
     $$\phi(C_{y_{ij}}) = C_{m_{ij}}(\phi(y_{ij})\phi(\bar{y}_{ij})) = C_{m_{ij}}(x_{y_{ij}}\bar{x}_{{y}_{ij}})e_{ii} = 0.$$
    Similarly, if $1 \leq i \not= j \leq N$, $y_{ij} \in {Y_{W^{\bullet}, S^{\bullet}, +}}$, $m_{ij} < \infty$, we have 
     $$\phi(C_{\bar{y}_{ij}}) = C_{m_{ij}}(\phi(\bar{y}_{ij})\phi(y_{ij})) = C_{m_{ij}}(\bar{x}_{{y}_{ij}}x_{y_{ij}})e_{jj} = 0.$$

    If $1 \leq i \leq N$ and $j = N +1$, we have $C_{y_{ij}} = y_{ij}\bar{y}_{ij} - [s_i]$ and 
    $$\phi(C_{y_{ij}})  =\phi(y_{ij})\phi(\bar{y}_{ij}) - e_{ii} = e_{i \underline{N + 1}}e_{\underline{N + 1} i} - e_{ii} = 0.$$
    Likewise, if $i= N +1$ and $1 \leq j \leq N$, we have 
    $$\phi(C_{y_{ij}})  = e_{\underline{N + 1} j}e_{j \underline{N + 1}} - e_{\underline{N+1}\  \underline{N+1}} = 0.$$
   Thus all generators of $\hat{I}_{W^{\bullet}, S^{\bullet}}$ are in $\ker \phi$ and hence $\phi(\hat{I}_{W^{\bullet}, S^{\bullet}}) = 0$. 
      \end{proof}

   \begin{lemma}\label{one} Let $\Phi : \hat{R}_{W^{\bullet}, S^{\bullet}} \to M_{N + 1}(Q)$ and $\Psi: M_{N + 1}(Q) \to \hat{R}_{W^{\bullet}, S^{\bullet}}$ be defined as above. Then $\Psi \circ \Phi$ is the identity map on $\hat{R}_{W^{\bullet}, S^{\bullet}}$.
   \end{lemma}
   
   \begin{proof}
   It is enough to show that this is true on the generators $[s_i]^{\sim}, s_i \in S$ and 
   $y_{ij}^{\sim},   y_{ij} \in Y_{W^{\bullet}, S^{\bullet}}$. 
   Let $s_i \in S$, then 
   $$\Psi \circ \Phi([s_i]^{\sim}) = \Psi(\phi([s_i]) = \Psi(e_{ii}) = \psi_i(1_Q)P_i^i  = [s_i]^{\sim}.$$
   If $y_{ij} \in Y_{W^{\bullet}, S^{\bullet}, +} , 1\leq i \not= j \leq N$, we have 
   $$
   \begin{aligned}
   \Psi \circ \Phi(y_{ij}^{\sim}) &= \Psi(\phi(y_{ij})) = \Psi(x_{y_{ij}}e_{ij})\\
    &= \psi_i(x_{y_{ij}})P_i^j 
    = P_i^iy_{ij}^{\sim}P_j^iP_i^j\\
     &= y_{ij}^{\sim}.
   \end{aligned}
   $$
   
   If $y_{ij} \in Y_{W^{\bullet}, S^{\bullet}, +} , 1\leq i \not= j \leq N$, we have 
   $$
   \begin{aligned}
   \Psi \circ \Phi(\bar{y}_{ij}^{\sim}) &= \Psi(\phi(\bar{y}_{ij})) = \Psi(\bar{x}_{y_{ij}}e_{ji})\\
    &= \psi_j(\bar{x}_{y_{ij}})P_j^i 
    = P_j^j\bar{y}_{ij}^{\sim}P_i^jP_j^i\\
     &= \bar{y}_{ij}^{\sim}.
   \end{aligned}
   $$

   Finally, if $i = N + 1$ or $j = N + 1$, then 
    $$
   \begin{aligned}
   \Psi \circ \Phi(y_{ij}^{\sim}) &= \Psi(\phi(y_{ij})) = \Psi(e_{ij})\\
    &= \psi_i(1_Q)P_i^j 
    = [s_is_j]^{\sim} = y_{ij}^{\sim}\\
   \end{aligned}
   $$
   \end{proof}

 \begin{lemma}\label{two}Let $\Phi : \hat{R}_{W^{\bullet}, S^{\bullet}} \to M_{N + 1}(Q)$ and $\Psi: M_{N + 1}(Q) \to \hat{R}_{W^{\bullet}, S^{\bullet}}$ be defined as above. Then $\Phi \circ \Psi$ is the identity map on $M_{N + 1}(Q)$.
   \end{lemma}

\begin{proof} It is enough to show that $\Phi \circ \Psi$ acts as the identity on the generators of $M_{N + 1}(Q)$ as a ring; 
$\{x_{y}e_{ij}, \bar{x}_{y}e_{ij}, e_{ij} | \ 1\leq i, j \leq N + 1, y \in Y_{W, S, +}\}$. 
For $1\leq i, j \leq N + 1$, we have 
$$
\begin{aligned}
\Phi(\Psi(e_{ij})) = \Phi(P_i^j)
&=\begin{cases}
\Phi([s_i]^{\sim}) &\mbox{if} \ i = j\\
\Phi((y_{i \underline{N +1}}y_{\underline{N+1}  j})^{\sim}) &\mbox{if} \ 1\leq i \not=j \leq N\\
\Phi(y_{\underline{N+1}  j}^{\sim}) &\mbox{if} \ i = N + 1 \ \mbox{and} \ 1\leq j \leq N\\
\Phi(y_{i \underline{N+1}}^{\sim}) &\mbox{if} \ j = N + 1 \ \mbox{and} \ 1\leq i \leq N\\
\end{cases}\\
&=\begin{cases}
e_{ii} &\mbox{if} \ i = j\\
e_{i \underline{N +1}}e_{\underline{N+1} j} = e_{ij}  &\mbox{if} \ 1\leq i \not=j \leq N\\
e_{\underline{N+1} j}  &\mbox{if} \ i = N + 1 \ \mbox{and} \ 1\leq j \leq N\\
e_{i \underline{N+1}}  &\mbox{if} \ j = N + 1 \ \mbox{and} \ 1\leq i \leq N\\
\end{cases}
\end{aligned}
$$
Thus $\Phi(\Psi(e_{ij})) = \Phi(P_i^j) = e_{ij}, 1 \leq i, j \leq N + 1$.

From the above calculation we have  that $\Phi(P_i^j) = e_{ij}, 1 \leq i, j \leq N + 1$. Thus if 
$1\leq i, j \leq N + 1, y \in Y_{W, S, +}$, we have 
$$
\begin{aligned}
\Phi\left(\Psi(x_{y}e_{ij})\right) &= \Phi\left(\psi_i(x_{y})P_i^j\right)\\
&=\Phi\left(P_i^{\dot{o}(y)}y^{\sim}P_{\dot{t}(y)}^iP_i^j\right) \\
&= e_{i\dot{o}(y)}\phi(y)e_{\dot{t}(y)i}e_{ij}\\
&= e_{i\dot{o}(y)}x_ye_{\dot{o}(y)\dot{t}(y)}e_{\dot{t}(y)i}e_{ij}\\
&= x_ye_{ij}.
\end{aligned}
$$

For 
$1\leq i, j \leq N + 1, y \in Y_{W, S, +}$, we have 
$$
\begin{aligned}
\Phi\left(\Psi(\bar{x}_{y}e_{ij})\right) &= \Phi\left(\psi_i(\bar{x}_{y})P_i^j\right)\\
&=\Phi\left(P_i^{\dot{o}(\bar{y})}\bar{y}^{\sim}P_{\dot{t}(\bar{y})}^iP_i^j\right) \\
&= e_{i\dot{o}(\bar{y})}\phi(\bar{y})e_{\dot{t}(\bar{y})i}e_{ij}\\
&= e_{i\dot{o}(\bar{y})}\bar{x}_ye_{\dot{o}(\bar{y})\dot{t}(\bar{y})}e_{\dot{t}(\bar{y})i}e_{ij}\\
&= \bar{x}_ye_{ij}.
\end{aligned}
$$
This concludes our proof.
\end{proof}

\section{Zero Divisors in $R_{W, S}$}\label{conclude}
In this section, we put together the results from the previous sections, to get an understanding of zero divisors in the ring $R_{W, S}$, which is the main topic of interest of the paper.  

\subsection{} We let $(W, S)$ be a Coxeter system with $S = \{s_i\}_{i \in I}$, not necessarily finite, and Coxeter matrix $(m_{ij})_{i, j \in I}$.
We use  the notation established in Section \ref{cox} throughout. For $p \in \hat{P}_{W, S}$, we let $x^{\sim} = p + \hat{I}_{W, S}$ denote 
the coset of $p$ in the quotient ring $\hat{R}_{W, S}$.

Although $\hat{R}_{W, S}$ is clearly not a domain, we  have the following result:

\begin{theorem}\label{main}Let  $(W, S)$ be a Coxeter system with $S = \{s_i\}_{i \in I}$, and 
Coxeter matrix $(m_{ij})_{i, j \in I}$.
Let 
$x_1 \in [s_i]{\hat{R}_{W, S}}[s_j]$ and $x_2 \in [s_j]{\hat{R}_{W, S}}[s_k]$,   be elements of ${\hat{R}_{W, S}}$ such that $x_1x_2 = 0$, 
then either $x_1 = 0$ or $x_2 = 0$.
\end{theorem} 

\begin{proof}
If $S$ is infinite, then $x_1 = p_1^{\sim}$ and $x_2 = p_2^{\sim}$, where $p_1, p_2 \in \hat{P}_{W, S}$ and  are linear combinations of paths involving only a finite number of vertices and edges. Thus, using the direct limit structure of $\hat{R}_{W, S}$ from Lemma \ref{dirlim}, we have 
 $x_1$ and $x_2$ are both in $\tilde{i}_{J, I}(\hat{R}_{W_J, S_J})$
 for some parabolic subgroup $(W_J, S_J)$ where $J \subset I$ is finite. 
 Since $\tilde{i}_{J, I}$ is a monomorphism, we can assume without loss of generality that 
$S$ is finite.

 When  $S$ is finite with cardinality $N$,  by Lemmas \ref{one} and \ref{two}, 
we have a  monomorphism  $\Psi : {\hat{R}_{W, S}} \to M_{N + 1}(Q)$
which maps $[s_i]^{\sim}{\hat{R}_{W, S}}[s_j]^{\sim}$ into $Qe_{ij}$. Furthermore,  $Q$ is a domain 
by  Theorem \ref{zerodiv}. Thus it is enough to show that 
if $(q_1e_{ij})(q_2e_{jk}) = 0$ where $q_1, q_2 \in Q$, then either $q_1 = 0$ or $q_2 = 0$. 
This is obvious since $Q$ is a domain. 
\end{proof}

\begin{corollary}Let let $(W, S)$ be a Coxeter system with $S = \{s_i\}_{i \in I}$, and 
Coxeter matrix $(m_{ij})_{i, j \in I}$.
Let 
$x_1 \in [s_i]R_{W, S}[s_j]$ and $x_2 \in [s_j]R_{W, S}[s_k]$  be elements of ${R_{W, S}}$ such that $x_1x_2 = 0$, 
then either $x_1 = 0$ or $x_2 = 0$.
\end{corollary} 

\begin{proof}By Lemma \ref{Rtohat}, we have  a monomorphism $\tilde{i}: R_{W, S} \to \hat{R}_{W, S}$ on the quotient algebras. Therefore the result follows from Theorem \ref{main}. 
\end{proof}

\end{document}